\documentclass[10pt]{amsart}
\usepackage{amsmath}
\usepackage{amssymb}
\usepackage{amsfonts}
\usepackage{mathrsfs}
\usepackage{amscd}
\usepackage[demo]{graphicx}
\usepackage[shortlabels]{enumitem}
\usepackage{mathtools}
\usepackage{tikz-cd}
\usepackage{cleveref}
\usepackage{pstricks-add}
\usepackage{pgf,tikz}
\usepackage{subcaption}
\usepackage[symbol]{footmisc}

\usetikzlibrary{arrows}

\crefformat{section}{\S#2#1#3} % see manual of cleveref, section 8.2.1
\crefformat{subsection}{\S#2#1#3}
\crefformat{subsubsection}{\S#2#1#3}

\numberwithin{equation}{section}       % Number formulas within sections

\theoremstyle{plain}
\newtheorem{theo}{Theorem}
\newtheorem{prop}{Proposition}[section]

\newtheorem{lemm}[prop]{Lemma}

\theoremstyle{definition}

\newtheorem{defi}[prop]{Definition}
\theoremstyle{remark}

\newtheoremstyle{citing}% name
  {3pt}%      Space above, empty = `usual value'
  {3pt}%      Space below
  {\itshape}% Body font
  {}%         Indent amount (empty = no indent, \parindent = para indent)
  {\bfseries}% Thm head font
  {.}%        Punctuation after thm head
  {.5em}%     Space after thm head: " " = normal interword space;
  {\thmnote{#3}}% Thm head spec
\theoremstyle{citing}
\newcommand{\N}{\mathbb{N}}

\newcommand{\R}{\mathbb{R}}

\newcommand{\cO}{\mathcal{O}}

\newcommand{\cS}{\mathcal{S}}

\newcommand{\tV}{\widetilde{V}}

\newcommand{\teta}{\widetilde{\teta}}

\renewcommand{\=}{ : = }
\DeclareMathOperator{\dist}{dist}

\DeclareMathOperator{\Crit}{Crit} % Set of critical points
\newcommand{\ulx}{\underline{x}}
\newcommand{\ula}{\underline{a}}
\newcommand{\uli}{\underline{i}}

\tikzset{
  declare function={
    sgn(\x) = (and(\x<0, 1) * -1) +
    (and(\x>0, 1) * 1) +
    (and(\x==0, 1) * 0);
  }
}

\usepackage{lipsum}
\begin{document}

\title{invariant measures for interval maps without lyapunov exponents}
\author{Jorge Olivares-Vinales} \thanks{This article was partially supported by ANID/CONICYT doctoral fellowship 21160715. } 
\address{Department of Mathematics, University of Rochester. Hylan Building, Rochester, NY 14627, U.S.A. \newline \indent
  Departamento de Ingenier\'ia Matem\'atica, Universidad de Chile, Beauchef 851, San-tiago, Chile.}
\email{jolivar2@ur.rochester.edu}

\begin{abstract}
  We construct an invariant measure for a piecewise analytic interval map whose Lyapunov exponent is
  not defined. Moreover, for a set of full measure, the pointwise Lyapunov exponent is not
  defined. This map has a Lorenz-like singularity and non-flat critical points.
\end{abstract}

\maketitle

\section{Introduction}
Lyapunov exponents play an important role in the study of the ergodic behavior of dynamical systems. In particular, in the seminal work of Pesin  (referred to as ``Pesin Theory''), the
existence and positivity of Lyapunov exponents were used to study the dynamics of non-uniformly hyperbolic systems, see for example \cite[Supplement]{HaKa95}. Using these
ideas, Ledrappier \cite{Led81} studied ergodic properties of absolutely continuous invariant measures for regular maps of the interval under the assumption that the Lyapunov
exponent exists and is positive. Recently Dobbs \cite{Dob14}, \cite{Dob15} developed the Pesin theory for noninvertible interval maps with Lorenz-like singularities and non-flat
critical points.  Lima \cite{Lim18} constructs a symbolic extension for these maps that code the measures with positive Lyapunov exponents.

In the case of continuously differentiable interval maps, Przytycki proved that ergodic invariant measures have nonnegative Lyapunov exponent, or they are supported on a strictly
attracting periodic orbit of the system. Moreover, there exists a set of full measure for which the pointwise Lyapunov exponent exists and is nonnegative, see \cite{Prz93},
\cite[Appendix A]{RL20}. 

In this paper, we show that the result above cannot be extended to continuous piecewise differentiable
interval maps with a finite number of non-flat critical points and Lorenz-like singularities. In 
particular,
we construct a measure for a unimodal map with a Lorenz-like singularity and two non-flat critical points for which the Lyapunov exponent does not exist.
Moreover, for this map, the pointwise Lyapunov exponent does not exist for a set of full measure. Thus, our example shows
that the techniques developed by Dobbs \cite{Dob14}, \cite{Dob15}, and Lima, \cite{Lim18}, cannot be extended to all maps
with critical points and Lorenz-like singularities.

% Thus our example shows that Pesin's theory cannot be applied to all interval
% maps with Lorenz-like singularities and non-flat critical points. 

Maps with Lorenz-like singularities are of interest since they appear in the study of
the Lorenz attractor, see \cite{GW79}, \cite{LuTu99}, and references therein. Apart from these motivations, these types of maps are of interest on their own since the
presence of these types of singularities create expansion, and hence enforce the chaotic behavior of the system, see \cite{ALV09}, \cite{LM13}, \cite{Dob14},  and references therein.

Additionally, the unimodal map that we consider has Fibonacci recurrence of the turning point 
(or just Fibonacci recurrence).
Maps with Fibonacci recurrence first appeared in the work of Hofbauer and Keller \cite{HK90} as possible interval maps having a wild attractor. Lyubich and Milnor \cite{LyMi93} proved that unimodal maps with a quadratic critical point and Fibonacci
recurrence do not only have any Cantor attractor but also have a finite absolutely continuous invariant measure, see also \cite{KN95}. Finally, Bruin, Keller,
Nowicki, and van Strien \cite{BKNvS96} proved that a $C^2-$unimodal interval map with a critical point of order big enough and with Fibonacci recurrence
has a wild Cantor attractor. On the other hand, in the work of 
Branner and Hubbard \cite{BH92}, in the case of complex cubic polynomials, and the work of Yoccoz, in the case of complex quadratic
polynomials, Fibonacci recurrence appeared  as the worst pattern of recurrence, see for example
\cite{Hub93}, and \cite{Mil92}.  Maps with Fibonacci recurrence also play an important role in the renormalization theory, see for example \cite{Sm07}, \cite{LeSw12}, \cite{GS18}, and
references therein. 

\subsection{Statement of results}
In order to state our main result, we need to recall some definitions.  A continuous map 
$f \colon [-1,1] \to [-1,1] $ is \emph{unimodal} if there is 
$c \in (-1,1)$ such that $f|_{[-1,c)}$ is increasing and $f|_{(c,1]}$ is decreasing. We call $c$ the \emph{turning point of} $f$. For every $A \subset [-1,1]$ and every $x \in [-1,1]$, we denote the \emph{distance from $x$ to $A$} by 
\[ \dist(x,A) \= \inf \{ |x-y| \colon y \in A \}. \] 
We will use $f'$ to denote the derivative of 
$f$. We will say that the point $c \in [-1,1]$ is a \emph{Lorenz-like singularity} if there exists $\ell^+$ and $\ell^-$ in $ (0,1)$, $L > 0$, and $\delta > 0$ such that the
following holds: For every $x \in (c, c + \delta)$
\begin{equation}
  \label{eq:right_lo-like}
  \frac{1}{L |x - c|^{\ell^+}} \leq |f'(x)| \leq  \frac{L}{ |x - c|^{\ell^+}},
\end{equation}
and for every $x \in (c - \delta,c)$
\begin{equation}
  \label{eq:left_lo-like}
  \frac{1}{L |x - c|^{\ell^-}} \leq |f'(x)| \leq \frac{L}{ |x - c|^{\ell^-}}.
\end{equation}
We call $\ell^+$ and $\ell^-$ the \emph{right and left order of} $c$ respectively.
For an interval map $f$, 
a point $\hat{c} \in [-1,1]$ is called a \emph{critical point} if $f'(\hat{c}) = 0$. We will say that a critical point $\hat{c}$ is \emph{non-flat} if there
exist $\alpha^+ > 0$, $\alpha^- >0$, $M >0$, and $\delta >0$ such that the following holds: \newline \noindent
For every $x \in (\hat{c}, \hat{c} + \delta)$
\begin{equation}
  \label{eq:right_crit_point}
  \left| \log \frac{|f'(x)|}{|x - \hat{c}|^{\alpha^+}} \right| \leq M,
\end{equation}
and for every $x \in (\hat{c} - \delta, \hat{c})$
\begin{equation}
  \label{eq:left_crit_point}
  \left| \log \frac{|f'(x)|}{|x - \hat{c}|^{\alpha^-}} \right| \leq M.
\end{equation}
We call $\alpha^+$ and $\alpha^-$ the \emph{right and left order of} $\hat{c}$ respectively.
Let us denote by $\Crit(f)$ the set of critical points of $f$. If $f$ is a unimodal map with turning point~$c$, we will use the notation
$\cS(f) \= \Crit(f) \cup \{ c \}$.
Let us denote by $C^{\omega}$ the class of analytic maps. Here
we will say that $f$ is a \emph{$C^{\omega}$-unimodal map} if it is of class $C^{\omega}$ outside $\cS(f)$. 

We denote the orbit of $x \in [-1,1]$ under $f$ by \[ \cO_f(x) \= \{ f^n(x) \colon n \geq 0 \}. \]
For a probability measure $\mu$ on $[-1,1]$ that is invariant by $f$, we define the 
\emph{pushforward of $\mu$ by $f$} as \[ f_*\mu \= h \circ f^{-1}.\] Denote by 
\[ \chi_{\mu}(f) \= \int \log |f'| d\mu,\] its \emph{Lyapunov exponent}, if the integral exists. Similarly, for every $x \in [-1,1]$, such that $\cO_f(x) \cap \cS(f) = \emptyset$, denote by
\[ \chi_f(x) \= \lim_{n \to \infty} \frac{1}{n} \log|(f^n)'(x)|, \] the \emph{pointwise Lyapunov exponent of $f$ at $x$}, if the limit exists.

Let $\lambda_F \in (0,2]$ be so that the map $T_{\lambda_F}\colon [-1,1] \to [-1,1]$, defined by
\begin{equation}
  \label{eq:tent_map}
 T_{\lambda_F}(x) \= \lambda_F(1 - | x |) -1,
\end{equation}
for every $x \in [-1,1]$, has Fibonacci recurrence and let $\mu_P$ be the unique measure that is
ergodic, 
invariant by $T_{\lambda_F}$, and supported on $\overline{\cO_{T_{\lambda_F}}(0)}$,  see \S \ref{FTM}.

\begin{theo}
  \label{thm:1}
  Let $h \colon [-1,1] \to [-1,1]$ be a homeomorphism of class $C^\omega$ on \newline $[-1,1]~\setminus\{0\}$ with a unique non-flat critical point at $0$, and put $\tilde{\mu}_P \= h_*\mu_P$.
  Then the $C^\omega$-unimodal map
  $f~\=~h~\circ~T_{\lambda_F}~\circ~h^{-1} $ has a Lorenz-like singularity at $\tilde{c} \= h(0)$ and
  is so that:
  \begin{enumerate}
  \item $\chi_{\tilde{\mu}_P}(f) $ is not defined.
  \vspace{0.3cm}
  \item For $x \in \overline{\cO_f(\tilde{c})}$, the pointwise Lyapunov exponent of $f$ at $x$ does
  not exist if $\cO_f(x) \cap \cS(f) = \emptyset$, and it is not defined if 
  $\cO_f(x) \cap \cS(f) \neq \emptyset$.
  %For every $x \in \overline{\cO_f(\tilde{c})}$, such that $\cO_f(x) \cap \cS(f) = \emptyset$, the pointwise Lyapunov
  %exponent $\chi_f(x)$ does not exist.
  \vspace{0.3cm}
  \item $\log(\dist(\cdot,\cS(f))) \notin L^1(\tilde{\mu}_P).$
  \vspace{0.3cm}
  \item $f$ has \emph{exponential recurrence} of the Lorenz-like singularity orbit, thus, 
        \begin{equation*}
            \limsup_{n \to \infty} \frac{ -\log |f^n(\tilde{c}) - \tilde{c}|}{n} \in (0, +\infty).
        \end{equation*}
  \end{enumerate}
\end{theo}

The map $f$ in Theorem \ref{thm:1} has a Lorenz-like singularity at $\tilde{c}$ and two non-flat critical points, given by
the preimages by $f$ of the Lorenz-like singularity $\tilde{c}$,
see Proposition \ref{prop:0}. Since $h$ is a homeomorphism of
class $C^{\omega}$ on $[-1,1] \setminus \{ 0 \}$, these are non-flat critical points of inflection type.

Dobbs constructed an example of a unimodal map with a flat critical point and singularities at the boundary, for which 
the Lyapunov exponent of an invariant measure does 
not exist, see \cite[Proposition 43]{Dob14}. 
For interval maps with infinite Lyapunov exponent see \cite[Theorem A]{Ped20}, and references therein.

The negation of item (3) in Theorem \ref{thm:1} is considered in several works as a regularity
condition to study ergodic invariant measures. In \cite{Lim18}, Lima studied measures satisfying this condition for
interval maps with critical points and discontinuities, he called measures satisfying this condition \emph{$f-$adapted}. 
By the Birkhoff ergodic theorem, if $\log(\dist(\cdot,\cS(f))) \in L^1(\mu)$, then for an ergodic 
invariant measure $\mu$, we have \[\lim_{n \to \infty} \frac{1}{n} \log(\dist(f^n(x),\cS(f))) = 0,\] $\mu-$a.e.
Ledrappier called measures satisfying this last condition \emph{non-degenerated}, for interval maps with a finite number
of critical points, see \cite{Led81}. The measure $\tilde{\mu}_P$ does not satisfy the non-degenerated condition.
For more results related to this condition see \cite{Lim18} and references therein.

For continuously differentiable interval maps with a finite number of critical points, every ergodic 
invariant measure that is not supported on an attracting periodic point satisfies 
$\lim_{n \to \infty} \frac{1}{n} \log(\dist(f^n(x),\cS(f))) = 0$, a.e., see \cite{Prz93} and 
\cite[Appendix]{RL20}. Item (3) in Theorem \ref{thm:1} tells us that we cannot extend this to
piecewise differentiable maps with a finite number of critical points and Lorenz-like singularities.

Item (4) in Theorem \ref{thm:1} stresses important information relative to 
the recurrence of the Lorenz-like singularity. This item represents a crucial difference between
smooth  interval maps, and the case of interval maps with critical points and Lorenz-like singularities. In the smooth
case, certain conditions
on the growth of the derivative restrict the recurrence to the critical set, see for example
\cite{CE83}, \cite{Tsu93}, \cite{GS14}, and references therein.
Using  the terminology in \cite{DPU96}, item (4) shows that Rule I is sharp for the map $f$ in Theorem \ref{thm:1}.
Finally, we have that the map $f$ in Theorem 
\ref{thm:1} satisfies Tsuji's weak regularity condition, due to the interaction between the critical points and
the Lorenz-like singularity.

\subsection{Example} Now we will provide an example of a map $f$ as in Theorem \ref{thm:1}. Fix $\ell^+$ and $\ell^-$ in $(0,1)$. Put $\alpha^+ \= \frac{1}{1-\ell^+}$ and
$\alpha^- \= \frac{1}{1-\ell^-}$. Define  \[h_{\alpha^+,\alpha^-} \colon [-1,1] \longrightarrow [-1,1]\]
as
\begin{equation}
  \label{eq:homeo}
  h_{\alpha^+,\alpha^-}(x) =
  \begin{cases}
    |x|^{\alpha^+} & \text{if $x \geq 0$} \\
    -|x|^{\alpha^-} & \text{if $x < 0$}. 
  \end{cases}
\end{equation}
So
\begin{equation}
  \label{eq:inv_homeo}
  h^{-1}_{\alpha^+,\alpha^-}(x) =
  \begin{cases}
    |x|^{1/\alpha^+} & \text{if $x \geq 0$} \\
    -|x|^{1/ \alpha^-} & \text{if $x < 0$}. 
  \end{cases}
\end{equation}
Then by \eqref{eq:tent_map}, \eqref{eq:homeo},\eqref{eq:inv_homeo}, and the chain rule, we have
\begin{eqnarray}
  f'(x) &=&\lambda_F  \frac{h'_{\alpha^+,\alpha^-} (T_{\lambda_F} (h^{-1}_{\alpha^+,\alpha^-}(x)))}{h'_{\alpha^+,\alpha^-}(h^{-1}_{\alpha^+,\alpha^-}(x))} \nonumber
\end{eqnarray}
for every $x \in [-1,1] \setminus \{ h_{\alpha^+,\alpha^-}(0) \}.$
The function $h'_{\alpha^+,\alpha^-} (T_{\lambda_F} (h^{-1}_{\alpha^+,\alpha^-}(x)))$ is bounded for $x$ close enough to $0$, see \S 5. Then, by \eqref{eq:homeo} and
\eqref{eq:inv_homeo}, there exists $L>0$
such that for every $x \in (h(0),h(\delta))$,
\[ \frac{1}{L |x|^{\ell^+}} \leq |f'(x)| \leq \frac{L}{|x|^{\ell^+}}, \]and for every $x \in (h(0),h(-\delta))$,
\[ \frac{1}{L |x|^{\ell^-}} \leq |f'(x)| \leq \frac{L}{|x|^{\ell^-}}. \]  Thus, $h(0)$ is a Lorenz-like singularity of $f$, see Figure~\ref{fig:1}. Also, by
\eqref{eq:homeo} and \eqref{eq:inv_homeo}, if $\delta$ is small enough so that $T^{-1}_{\lambda_F}(0) \cap (-\delta, \delta) = \emptyset$, the two critical points
of $f$ are non-flat. The one to the left of $h(0)$ has right order $\alpha^+$ and left order $\alpha^-$, and the one to the right of $h(0)$ has right order $\alpha^-$ and left
order $\alpha^+.$

\begin{figure}[h!]
  \centering
  \begin{subfigure}[b]{.3 \textwidth}
    \begin{tikzpicture}[line cap=round,line join=round,>=triangle 45,x=1.6cm,y=1.6cm]
      \clip(-1.3,-1.3) rectangle (1.3,1.3);
      \draw (1,1)-- (-1,1);
      \draw (-1,1)-- (-1,-1);
      \draw (-1,-1)-- (1,-1);
      \draw (1,-1)-- (1,1);
      % \draw[dash] (0.0) {(0,0)} -- (1,1)  {(1,1)}
      \draw [dash pattern=on 1pt off 1pt] (-1,-1)-- (1,1);
      \draw (-1.2,1.08) node[anchor=north west] {1};
      \draw (1,-1.) node[anchor=north west] {1};
      \draw (-1.2,-1.) node[anchor=north west] {-1};
      \draw[line width=0.8,smooth,samples=100,domain=-1.0:1.0] plot(\x,{1.73*(1 - abs(\x)) -1});
    \end{tikzpicture}
    \caption{} 
  \end{subfigure}
  \begin{subfigure}[b]{.3 \textwidth}
    \begin{tikzpicture}[line cap=round,line join=round,>=triangle 45,x=1.6cm,y=1.6cm]
      \clip(-1.3,-1.3) rectangle (1.3,1.3);
      \draw (1,1)-- (-1,1);
      \draw (-1,1)-- (-1,-1);
      \draw (-1,-1)-- (1,-1);
      \draw (1,-1)-- (1,1);
      \draw [dash pattern=on 1pt off 1pt] (-1,-1)-- (1,1);
      \draw (-1.2,1.08) node[anchor=north west] {1};
      \draw (1,-1.) node[anchor=north west] {1};
      \draw (-1.2,-1.) node[anchor=north west] {-1};
      \draw[line width=0.8,smooth,samples=100,domain=-1.0:0.0] plot(\x,{-abs(\x)^(1.5)});
      \draw[line width=0.8,smooth,samples=100,domain=.0:1.0] plot(\x,{abs(\x)^(4)});
    \end{tikzpicture}
    \caption{}
  \end{subfigure}
  \begin{subfigure}[b]{.3 \textwidth}
    \begin{tikzpicture}[line cap=round,line join=round,>=triangle 45,x=1.6cm,y=1.6cm]
      \clip(-1.3,-1.3) rectangle (1.3,1.3);
      %\draw (-1,-1)-- (1,1);
      %\draw[line width=0.8, smooth,samples=3000,domain=-1.0:1.0] plot(\x,{(1.73*(1-(abs(\x))^(1/3))-1)^(3)});
      \draw[line width=0.8, smooth,samples=100,domain=0.18:1.0] plot(\x,{0-abs(1.73*(1-abs((\x))^(1/2))-1)^1.2});
      \draw[line width=0.8, smooth,samples=100,domain=0.0:0.18] plot(\x,{abs(1.73*(1-abs((\x))^(1/2))-1)^2});
      \draw[line width=0.8, smooth,samples=100,domain=-1.0:-0.3550] plot(\x,{0-abs(1.73*(1-abs((\x))^(1/1.2))-1)^1.2});
      \draw[line width=0.8, smooth,samples=100,domain=-0.345:0.0] plot(\x,{abs(1.73*(1-abs((\x))^(1/1.2))-1)^2});
      \draw (1,1)-- (-1,1);
      \draw (-1,1)-- (-1,-1);
      \draw (-1,-1)-- (1,-1);
      \draw (1,-1)-- (1,1);
      \draw [dash pattern=on 1pt off 1pt] (-1,-1)-- (1,1);
      %\draw [dash pattern=on 1pt off 1pt] (0,-1)-- (0,1);
      \draw (-1.2,1.08) node[anchor=north west] {1};
      \draw (1,-1.) node[anchor=north west] {1};
      \draw (-1.2,-1.) node[anchor=north west] {-1};
    \end{tikzpicture}
    \caption{}
  \end{subfigure}
  \caption{Graphics of the functions $T_{\lambda_F}(x)$ (Figure (A)), $h_{\alpha}(x)$ for $\alpha^+= 2$ and $\alpha^- = 1.2$ (Figure (B)) and $f(x)$ (Figure (C)).}
  \label{fig:1}
\end{figure}
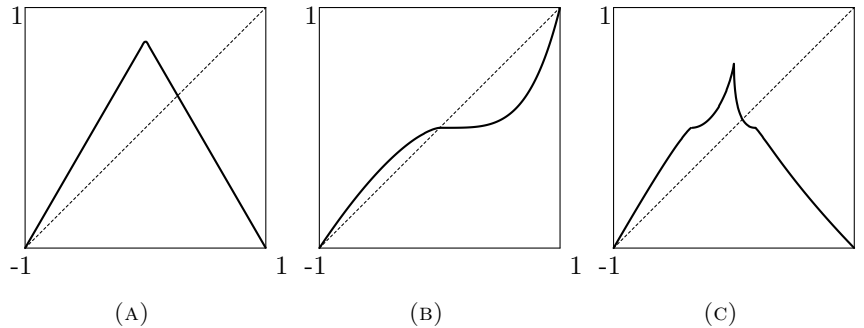

\subsection{Strategy and organization}
We now describe the strategy of the proof of Theorem \ref{thm:1} and the organization of the paper.

In \S 2 we review some general theory and results concerning the kneading sequence for unimodal maps and unimodal maps with Fibonacci recurrence. In particular, in \S 2.1 we will
describe the relationship between the kneading map and the kneading sequence, and in \S 2.2 we define
the Fibonacci recurrence. These two elements will be of importance to describe the
combinatorics of the critical orbit.

In \S 3.1  we make a detailed description of the set $\overline{\cO_{T_{\lambda_F}}(0)}$, and following \cite{LyMi93}, we construct a partition of it that
will allow us to estimate close return times to the turning point and lower bounds for the distances of these close returns. In \S 3.2 we estimate how fast the orbit of the turning
point return to itself in terms of the return time (see Lemma \ref{lemm:diam_estimate}). This estimation is of importance since it gives us an exact estimation of the growth of the
geometry near the turning point for our map $f$.

In \S 4 we describe the unique ergodic invariant measure $\mu_P$ supported on
$\overline{\cO_{T_{\lambda_F}}(0)}$, restricted to the partition constructed in
\S 3.1. We need this estimation to prove part (1) in Theorem \ref{thm:1}.

In section \S 5 we prove the following proposition that will give us a key bound on the derivative of $f$ in terms of $h^{-1}$. Without loss of generality, we will assume that $h$ preserves orientations,

\begin{prop}
  \label{prop:0}
  Let $h$ and $f$ be as in Theorem \ref{thm:1}. Then $f$ has a Lorenz-like singularity at
  $\tilde{c}$. Moreover there exists $\alpha^+> 1$, $\alpha^- >1$, $K > 0$, and $\delta > 0$ such that the following property holds:
  For every  $x~\in~(\tilde{c} , h(\delta)),$  
  \begin{equation}
    \label{eq:thm_1}
    K^{-1}|h^{-1}(x)|^{-\alpha^+} \leq |f'(x)| \leq K |h^{-1}(x)|^{-\alpha^+}, 
  \end{equation}
  and for every $x \in (h(-\delta) , \tilde{c}),$
  \begin{equation}
    \label{eq:thm_2}
    K^{-1}|h^{-1}(x)|^{-\alpha^-} \leq |f'(x)| \leq K |h^{-1}(x)|^{-\alpha^-}. 
  \end{equation}
\end{prop}

In \S 6 we prove the following proposition that implies items (1), and (3) in Theorem~\ref{thm:1},
\begin{prop}
  \label{prop:1}
  Let $h$ and $f$ be as in Theorem \ref{thm:1}. Then~
  \begin{itemize}
      \item[(i)] $\int \log |f'|^+ d\tilde{\mu}_P = +\infty$,
      \vspace{0.2 cm}
      \item[(ii)] $\int \log |f'|^- d\tilde{\mu}_P= -\infty$, and
      \vspace{0.2cm}
      \item[(iii)] $\int | \log(\dist(\cdot, \cS(f)))| \,  \tilde{\mu}_P = +\infty$.
  \end{itemize}
\end{prop}

To prove the first part of Proposition \ref{prop:1}, we use the fact that around the Lorenz-like singularity, the geometry of $f$ grows at the same rate as the measure decreases.
This implies that in a sequence of disjoint intervals that converges to the critical point, the integral of $\log |f'|$ is bounded from below by a positive constant. For the second
part, we use the fact that the two preimages of the turning point of $f$ are critical points and both belong to the set $\overline{\cO_f(\tilde{c})}$. The third part of the proposition is a consequence of the 
estimation that we get in the proof of the first part.

In section \S 7 we prove the following proposition, that along with the fact that $f$ is transitive on 
$\overline{\cO_f(\tilde{c})}$, will imply item (2) in Theorem \ref{thm:1}. Recall that 
for $x \in \overline{\cO_f(\tilde{c})}$ such that $\tilde{c} \in \cO_f(x)$ we have that the pointwise Lyapunov exponent is not
defined, since for $n$ large enough $\log |(f^n)'(x)|$ is not defined.

\begin{prop}
  \label{prop:2}
  Let $h$ and $f$ be as in Theorem \ref{thm:1}. Then for every $x \in \overline{\cO_f(\tilde{c})}$
  with $\tilde{c} \notin \cO_f(x)$, we have that 
  \begin{equation}
     \label{eq:prop_2}
     \liminf_{n \to \infty} \frac{1}{n}\log|(f^n)'(x)| \leq \left( 1 - \frac{\alpha}{\varphi} \right)
     \log \lambda < \log \lambda  \leq \limsup_{n\to \infty} \frac{1}{n}\log|(f^n)'(x)|.
    \end{equation}
\end{prop}

To prove Proposition \ref{prop:2}, we use the fact that $f$ restricted to the set $\overline{\cO_f(\tilde{c})}$
is minimal, so the orbit of every point accumulates points far from the turning
point. In that case, the derivative is bounded so the limit of that subsequence must be the same as the one in $T_{\lambda_F}$.
On the other hand, if we look at a subsequence that accumulates at the Lorenz-like singularity, the growth of the derivative is exponential with respect to the return time, so the limit
of this subsequence will be bounded away from zero.

In section \S 8 we will prove the following proposition that implies item (4) in Theorem \ref{thm:1}. 
Since $f$ is topologically conjugated to the Fibonacci tent map we know that $h(0)$ is 
recurrent, and that the recurrence times are given by the Fibonacci numbers. Then, to have an estimate on the recurrence 
of the turning point it is enough to estimate the decay of the distances 
$|f^{S(k)}(\tilde{c}) - \tilde{c}|$, where 
\[ S(0) = 1, \, S(1) = 2,\, S(2) = 3,\, S(3) = 5, \ldots,  \] are the Fibonacci numbers.
\begin{prop}
    \label{prop:3}
    There exist $\Theta$, $\alpha'$, $\alpha''$ positive numbers, such that 
    \begin{equation}
        \lambda^{-S(k)\alpha''}\Theta^{-1} \leq |f^{S(k)}(\tilde{c}) - \tilde{c}| \leq
        \lambda^{-S(k)\alpha'} \Theta,
\end{equation}
for every $k \geq 1$.
\end{prop}
To prove Proposition \ref{prop:3} we estimate the diameter of certain symmetric intervals, whose closure are disjoint
from the Lorenz-like singularity, and whose lengths approximate the
left and right distance of the closest returns to the Lorenz-like singularity. To do this we use the mean 
value theorem, the fact that $h$ has a non-flat critical point at $0$, and Lemma \ref{lemm:diam_estimate} that give us
an estimate on the diameter of the preimage of these intervals. The reason to use these intervals is because when we try
to make a direct estimation the distance $|f^{S(k)}(\tilde{c}) - \tilde{c}|$ we do not have control on
how close to zero is the derivative of $h$.

\subsection{Acknowledgements}
The author would like to thanks Juan Riverla-Letelier for helpful discussions and encouragement for this work. The author is also grateful to Daniel Coronel and Yuri Lima for helpful comments. 
%%% Local Variables:
%%% mode: latex
%%% TeX-master: "Unbounded_Lyapunov_exponent"
%%% End:

\section{Preliminaries}

Throughout the rest of this work, we will denote by $I$ the closed interval $[-1,1] \subset \R$. We use $\N$ to denote the set of integers that are greater than or equal to
$1$ and put $\N_0 \= \N \cup \{ 0 \}$.

We endow $I$ with the distance induced by the absolute value $|\cdot|$ on $\R$. For $x \in \R$ and $r > 0$, we denote by $B(x,r)$ the open ball of $I$
with center at $x$ and radius $r$. For an interval $J \subset I$, we denote by $|J|$ its length.

For real numbers $a,b$
we put $[a,b] \= [\min\{a,b\}, \max\{a,b\}]$ in the same way $(a,b) \= (\min\{a,b\} , \max\{a,b\})$.

\subsection{The kneading sequence} Following \cite{MevS93}, we will introduce  the kneading invariant of a unimodal map and related properties.
Let $f \colon I \to I$ be a unimodal with turning point $c \in (0,1)$.
We will  use the notation $c_i \= f^i(c)$ for $i \geq 1$. 
Suppose $f(-1) = f(1) = -1$  and $c_2 < c < c_1.$ Let $\Sigma \= \{ 0,1,c \}^{\N_0}$ be the space of
sequences $\ulx =(x_0, x_1, x_2, \ldots)$. In $\Sigma$ we consider the topology generated by the cylinders \[  [a_0a_1\ldots a_{n-1}]_k \= \{ \ulx \in \Sigma \colon x_{k+i} = a_i
  \text{ for all } i = 0,1, \ldots , n-1\}.  \] With this topology, $\Sigma $ is a compact space.
% and the shift transformation defined as
% \begin{align*}
%   \sigma \colon \Sigma & \longrightarrow \Sigma \\
%   x_0, x_1, x_2, \ldots  &\longmapsto  x_1,x_2,x_3 \ldots,
% \end{align*}
% %%%$\sigma \colon \Sigma \to \Sigma$, defined by $\sigma(x_0, x_1, x_2, \ldots) = (x_1,x_2, \ldots)$,%%%
% is continuos.
Let us define %%%Let $\uli \colon I \to \Sigma$ be defined by $\uli (x) =  (i_0(x),i_1(x), \ldots)$%%%
\begin{align*}
  \uli \colon I & \longrightarrow \Sigma\\
  x & \longmapsto   (i_0(x),i_1(x), \ldots)
\end{align*}
where %%%$i_n(x) = 0$ if $f^n(x) \in [0,c) $, $i_n(x) = 1 $ if $f^n(x) \in (c,1]$ or $i_n(x) = c$ if $f^n(x) = c.$%%%
\[ i_n(x) =
  \begin{cases}
    0 & \quad \text{if } f^n(x) \in [-1,c) \\
    1 & \quad \text{if } f^n(x) \in (c,1] \\
    c & \quad \text{if } f^n(x) = c.
  \end{cases}
\]
The sequence $\uli (x)$ is called \emph{the itinerary of $x$ under $f$}. 
Given $n \in \N$ and $x \in I$ there exists $\delta >0$ such that $i_n(y) \in \{0,1\} $ and is constant for every $y \in (x, x+ \delta)$. Observe that this
value is not the same as $i_n(x)$ if $x$ is the turning point. It follows that
\[ \uli(x^+) \= \lim_{y \downarrow x } \uli(y) \hspace{20pt} \text{ and } \hspace{20pt} \uli(x^-) \= \lim_{y \uparrow x } \uli(y) \]
always exist.
Notice that $\uli(x^-)$ and $\uli(x^+)$ belong to $\{ 0,1 \}^{\N_0}$. The sequence $e_1,e_2, e_3, \ldots$ defined by $e_j\= i_j(c_0^+)$ is called the \emph{kneading invariant} of $f$.
A sequence $\ula \in \{0,1\}^{\N}$ is \emph{admisible} if there exists a unimodal map $f \colon I \to I$ with kneading invariant $\ula$.

We say that $Q \colon \N \to \N_0$ defines a \emph{kneading map} if $Q(k) < k$ for all $k \in \N$ and \[ (Q(j))_{k < j < \infty} \geq  (Q(Q(Q(k)) + j - k))_{k<j<\infty}  \] for all
$k$ with $Q(k)> 0$ ($\geq$ is the lexicographical order). A kneading map leads to an admisible kneading sequence in the following way: defined the sequence $S \colon \N_0 \to \N$ by
$ S(0) = 1$ and $S(k) = S(k-1) + S(Q(k))$ for $k \geq 1$. The kneading sequence $\{e_j\}_{j\geq 1}$ associated to $Q$ is given by $e_1 = 1$ and the relation
\begin{equation}
  \label{eq:cut_time}
  e_{S(k-1)+1}e_{S(k-1)+2} \ldots e_{S(k)-1}e_{S(k)} = e_1e_2 \ldots e_{S(Q(k))-1}(1 - e_{S(Q(k))}),
\end{equation}
for $k \geq 1$. The length of each string in (\ref{eq:cut_time}) is $S(Q(k))$, thus at the $kth$-step of the process we can construct $S(Q(k))$ symbols of the sequence. Since for
every $k \geq 1$, we have $Q(k) < k$, we get that $Q(1) = 0$. So, for $k = 1$, each string in (\ref{eq:cut_time}) has $1$ symbol. Then
\[ e_2 = e_{S(0)+1} = 1 - e_{S(0)}= 0.\] Hence, \[ c_2 < c < c_1.  \] 

\subsection{The Fibonacci tent map.}\label{FTM} %Let \[ S(0) = 1, \, S(1) = 2,\, S(2) = 3,\, S(3) = 5, \ldots,\] be the Fibonacci numbers. 
We will say that a unimodal map $f$ has \emph{Fibonacci recurrence} or it is a 
\emph{Fibonacci unimodal map}
% if \[ |c_{S(k)} - c| > |c_{S(k+1)} - c| \] for every $n \in \N_0$, so that 
%  \begin{equation}
%    \label{def:fibonacci}
%    |c_{S(0)} - c| < |c_{S(1)} - c|  < \ldots |c_{S{n}} - c| < |c_{S(n+1)} - c|< \ldots,  
%  \end{equation}
% and if $|c_3 - c| < |c_4 - c|$.
if the kneading map
associated to it is given by $Q(1) = 0$ and $Q(k) = k-2$ for $k > 1$. So the sequence $\{ S(n)\}_{n \geq 0}$ is given by the Fibonacci numbers
\[ S(0) = 1, \, S(1) = 2,\, S(2) = 3,\, S(3) = 5, \ldots.  \]  For a Fibonacci unimodal map $f$ we have that

\begin{equation}
  \label{def:fibonacci}
  |c_{S(0)} - c| > |c_{S(1)} - c|  > \ldots |c_{S(n)} - c| > |c_{S(n+1)} - c|> \ldots,  
\end{equation}
and
\begin{equation}
  \label{def:fibonacci_2}
  |c_3 - c| < |c_4 - c|.
\end{equation}
See \cite[Lemma 2.1]{LyMi93} and references therein.
The set $\overline{\cO_f(c)}$ is a Cantor set and the restriction of $f$ to this set is minimal and uniquely ergodic, see \cite[Proposition~1]{Br03} or \cite[Proposition~4]{CRL10} 
and references therein. 
The kneading invariant for a Fibonacci unimodal map starts like \[ 100111011001010011100 \ldots  \]
Let us consider the tent family $T_S\colon I \to I$ defined by $T_s(x) = s(1 - | x |) -1$ for every $x \in I$ and every $s \in (0,2]$. This family is \emph{full}, thus for every
kneading map $Q$ there is a parameter $s \in (0,2]$ so that the kneading map of $T_s$ is $Q$, see \cite{MiTh88}, \cite[Chapter 2]{MevS93}. So there exists $\lambda_F \in (0,2]$ such
that the kneading map associated to $T_{\lambda_F}$ is given by $Q(k) = \max\{0, k-2\}$. 

From now on we use the notations $T \= T_{\lambda_F}$, $\lambda \= \lambda_F$, $c \= 0$, and
$c_i \= T^i(c)$.
%%% Local Variables:
%%% mode: latex
%%% TeX-master: "Unbounded_Lyapunov_exponent"
%%% End:

\section{The set $\overline{\cO_T(c)}$}
\label{sec:post-critical}
\subsection{The combinatorics of the set $\overline{\cO_T(c)}$}
In this section we will give an explicit description of the set $\overline{\cO_{T(c)}}$  following \cite{LyMi93}.

Put  $S(-2) = 0$ and $S(-1) = 1$.
From (\ref{eq:cut_time}) we obtain that for every $k \geq 0$ the points $c_{S(k)}$ and $c_{S(k+2)}$ are on opposite sides of $c$. Since \[ c_{S(1)} = c_2 < c < c_1 =c_{S(0)}, \]
we conclude that for $k \equiv 0 \pmod 4,~c_{S(k)}$ is to the right of $c$ and if $k \equiv 2 \pmod4 $,~$c_{S(k)}$ is to the left of $c$. Since we also know that $c_{S(1)}$ is to
the left of $c$, we can conclude that for $k \equiv 1 \pmod 4$, $c_{S(k)}$ is to the left of $c$, and for $k \equiv 3 \pmod 4$, $c_{S(k)}$ is to the right of $c$. From this, we can
conclude that if $k$ is even the points $c_{S(k)}$ and $c_{S(k+1)}$ are in opposite sides of $c$, and therefore \[ [c_{S(k+1)}, c_{S(k)}]  \supseteq [c_{S(k+2)} , c_{S(k)}].  \] In the
case that $k$ is odd, $c_{S(k)}$ and $c_{S(k+1)}$ are on the same side with respect to $c$, and therefore \[ [c_{S(k+1)} , c_{S(k)}] \subseteq [c_{S(k+2)} , c_{S(k)}].  \]

For each $k \geq 0$ let $I_k$ be the smallest closed interval containing all of the points $c_{S(l)}$ for every $l \geq k$. For each $n \geq 0$ define $I^n_k \= T^n(I_k)$. By the
above discussion  
\begin{equation}
  \label{eq:I_K}
  I_k =
  \begin{cases}
    [c_{S(k)} , c_{S(k+1)}] \text{ if  $k$ is even, } \\
    [c_{S(k)}, c_{S(k+2)}]  \text{ if $k$ is odd. }
  \end{cases}
\end{equation}

\begin{lemm}
  \label{lemm:T_injective}
  For every $k \geq 1$, we have that $T^j$ is injective on $[c_1,c_{S(k)+1}]$. In particular $I_k^{j+1} = [c_{j+1},c_{S(k) + 1 + j}]$ for every $j \in \{1, \ldots, S(k-1)-1 \}.$ 
\end{lemm}
\begin{proof}
Since $|c - c_{S(k)}| > |c - c_{S(m)}|$  and $|T([c_{S(k)},c])| = \lambda|c - c_{S(k)}|$, for every $0 \leq k < m$ we get that $c_{S(k) +1} < c_{S(m) + 1} < c_1$, in particular
$I^1_k = [c_1, c_{S(k) +1}]$. In the case $k \geq 1$, by
(\ref{eq:cut_time}), with $k$ replaced by $k+1$, for every $j \in \{ 1, \ldots , S(k-1)-1 \}$ we have that $c_{S(k)+j}$ and $c_j$ are in the same side respect to $c$. Thus
$c \notin [c_{S(k) + j}, c_j] = T^{j-1}[c_{S(k) +1} , c_1]$, and then the map $T^j$ is injective on $[c_1, c_{S(k) +1}]$. In particular, for $1 < j \leq S(k-1)$

\begin{eqnarray}
  \label{eq:interval_image_general}
  I^{j}_k & = & T^{j-1}([c_1, c_{S(k)+1}])  \\
        & = & [c_j, c_{S(k) + j}] \nonumber 
        %& = & [c_{S(k-1)}, c_{S(k+1)}]. \nonumber
\end{eqnarray}
\end{proof}
Note that for $k \geq 1$, by Lemma \ref{lemm:T_injective}, with $j = S(k-1) -1$
\begin{equation}
  \label{eq:interval_image}
  I_k^{S(k-1)} = [c_{S(k-1)}, c_{S(k) + S(k-1)}] = [c_{S(k-1)}, c_{S(k+1)}].
\end{equation}
Then, by (\ref{eq:cut_time}), $c \in I^{S(k-1)}_k$ and $c \notin I^n_k  $ for every $0 < n < S(k-1)$.
\begin{lemm} 
  \label{lemm:orbit_order}
  For all $k  \geq 0$ we have that \[ |c_i - c| > |c_{S(k-1)} - c|,\] for all $ 0 < i < S(k)$, with $i \neq S(k-1).$ 
\end{lemm}

\begin{proof}
  We will use induction on $k$. The cases $k = 0\text{ and } 1$ are vacuously true, the cases $k = 2 \text{ and } 3$ are true by the definition of Fibonacci map and
  (\ref{def:fibonacci_2}).
  Suppose now that it is true for $k$. We will prove that is
  true for $k+1$. \newline
  \textbf{Case 1:}
  Since \[ |c_{S(k-1)} - c| > |c_{S(k)} - c|,\] we have that  \[ |c_i - c| > |c_{S(k)} - c|,  \] for all $0 < i < S(k)$.\newline
  \textbf{Case 2:} 
  Since \[ c_{S(k-1) +1} < c_{S(k) +1} < c_1 \] and $T^i$ is injective on $[c_{S(k-1) +1},c_1]$ for $0<i<S(k-2)$ by Lemma \ref{lemm:T_injective}, we have that
  $c_{S(k) +i} \in (c_{S(k-1) + i}, c_i)$, for $0 < i < S(k-2)$. By the induction hypothesis, \[ |c_i -c|>|c_{S(k-1)}-c|>|c_{S(k)} -c| \]
  and \[ |c_{S(k-1) + i} - c|>|c_{S(k-1)} - c|>|c_{S(k)} - c|, \] for $0 < i < S(k-2).$ By \ref{eq:cut_time}
  $c_i$ and $c_{S(k-1) +i}$ lie on the same side of $c$ for $0 < i < S(k-2)$.  The above implies
  that \[ |c_{S(k) + i} - c| > |c_{S(k)} - c| \] for all $0 < i < S(k-2)$.\newline
  \textbf{Case 3:}
  Since $T^{S(k-2) -1}$ is injective on $[c_{S(k-1)+1}, c_1]$ we get that $c_{S(k)+S(k-2)} \in (c_{S(k)}, c_{S(k-2)})$. Also, by (\ref{eq:cut_time}),  $c_{S(k) +S(k-2)}$ and $c_{S(k-2)}$ lie on
  the same side of $c$, and opposite to $c_{S(k)}$. Then \[ |c_{S(k) + S(k-2)} - c|<|c_{S(k-2)} - c|.\]
  Hence \[ c_{S(k-2) + 1} < c_{S(k) + S(k-2) +1} < c_1. \] So by Lemma \ref{lemm:T_injective}, $c_{S(k) + S(k-2) +i} \in (c_{S(k-2)+i} , c_i)$ for
  $0 < i < S(k-3)$. By the induction hypothesis \[ |c_i - c|>|c_{S(k)} - c| \hspace{10pt} \text{ and } \hspace{10pt} |c_{S(k-2) + i} - c|>|c_{S(k)} - c|, \] for $0<i<S(k-3).$ Since, by
  (\ref{eq:cut_time}), $c_{S(k-2) +i}$ and $c_i$ lie on the same side of $c$ for $0<i<S(k-3)$ we get \[ |c_{S(k) + S(k-2) +i} -c| > |c_{S(k)} - c|, \] for all $0<i<S(k-3).$ \newline
  \textbf{Case 4:} It remains to prove that \[ |c_{S(k) + S(k-2)} - c|>|c_{S(k)} - c|.\]
  Suppose by contradiction that \[ |c_{S(k) + S(k-2)} - c|<|c_{S(k)} - c|.\] Then $c_{S(k) +1} < c_{S(k) + S(k-2) +1} < c_ 1$. Since $T^{S(k-3) -1}$ is injective on
  $[c_{S(k) +1}, c_1]$ we get that $c_{S(k+1)} \in (c_{S(k-3)},c_{S(k) + S(k-3)})$. Noting that by (\ref{eq:cut_time}) $c_{S(k-3)}$ and $c_{S(k) + S(k-3)}$ are in the same side with respect to
  $c$ we have either \[|c_{S(k+1)}-c|>|c_{S(k) + S(k-3)}-c|>|c_{S(k)}-c| \] or \[|c_{S(k+1)}-c|>|c_{S(k-3)}-c|>|c_{S(k)}-c|, \]
  a contradiction. So we must have \[ |c_{S(k) + S(k-2)}-c|>|c_{S(k)}-c|, \] and this conclude the proof.
 \end{proof}
 Let us denote \[ J_k \= I_{k+1}^{S(k-1)} = [c_{S(k-1)}, c_{S(k+1) + S(k-1)}],\] and put \[ D_k \= [c,c_{S(k)}] \] for every $k \geq 1$. For every $n \geq 0$ we use the notation
 \begin{equation}
   \label{eq:J_k_definition}
   J^n_k \= T^n(J_k) = I^{S(k-1) + n}_{k+1}.
\end{equation}
 Note that by definition $D_{k'} \subset [c_{S(k)},c_{S(k+2)}]$, for every $k' \geq k \geq 1.$
 \begin{lemm}
  \label{lemm:J_k}
  For all $0< k <k'$ we have $J_{k'} \subset D_{k'-1} \subset I_{k}$ and $J_{k} \cap J_{k'} = \emptyset$.
\end{lemm}
\begin{proof}
  First we will prove that $J_{k+1}$ is contained in $D_k$ and $c \notin J_{k+1}$ for every $k \geq 0$. Fix $k \geq 0$.
  By (\ref{eq:cut_time}) with $k$ replaced by $k+3$ we have $c_{S(k)}$ and $c_{S(k+2) + S(k)}$ are in the same side of $c$. Since \[ |c_{S(k+2)} - c| < |c_{S(k+1)} -c|,\] we have
  \[c_{S(k+1) +1}<c_{S(k+2) +1}<c_1.\] By Lemma \ref{lemm:T_injective}, $T^{S(k) -1}$ is injective on $I_{k+1}^1$. Then, $c_{S(k+2) + S(k)} \in (c_{S(k+2)} , c_{S(k)})$. By (\ref{eq:cut_time})
  with $k$ replaced by $k+3$, we thus conclude $c_{S(k+2) + S(k)} \in (c_{S(k)} , c)$. Then
  \[J_{k+1} = [c_{S(k)}, c_{S(k) + S(k+2)}] \subset [c, c_{S(k)}] = D_k\subset [c_{S(k)},c_{S(k+2)}] \subseteq I_k \] and $c \notin J_{k+1}$. Since, by definition, for every $k' > k$ we have
  $D_{k'-1} \subset I_{k'-1} \subset I_k$, we get \[ J_{k'} \subset D_{k'-1} \subset I_k.  \]
  Now we will prove that $J_{k+1}$ and $I_{k+1}$ are disjoint.
  If $k$ is even then $I_{k+1} = [c_{S(k+1)} , c_{S(k+3)}]$. Since $c_{S(k)}$ and $c_{S(k+1)}$ lie on opposite sides respect to $c$, we have that $c_{S(k)}$, $c_{S(k) + S(k+2)}$, and
  $c_{S(k+3)}$ lie on the same side of $c$. By (\ref{lemm:orbit_order}), with $k$ replaced by $k+4$, we get \[ |c_{S(k)} - c|>|c_{S(k) + S_{k+2}} - c|>|c_{S(k+3)} - c|. \] So
  $J_{k+1} \cap I_{k+1} = \emptyset.$ Now, if $k$ is odd $I_{k+1} = [c_{S(k+1)}, c_{S(k+2)}]$ and $c_{S(k)}$, $c_{s(k+1)}$ and $c_{S(k) + S(k+2)}$ lie on the same side of $c$. Suppose that
  $|c_{S(k+1)} - c| > |c_{S(k) + S(k+2)} - c|$, then \[[c_{S(k+1)},c_{S(k) + S(k+2)}] \subset [c_{S(k)},c_{S(k) + S(k+2)}].\] Since $T^{S(k-1)}$ is injective on $[c_{S(k)},c_{S(k) + S(k+2)}]$, then
  $T^{S(k-1)}$ is injective on \newline $[c_{S(k+1)},c_{S(k) + S(k+2)}]$. So we get \[T^{S(k-1)}([c_{S(k+1)},c_{S(k) + S(k+2)}]) = [c_{S(k+1)+S(k-1)},c_{S(k+3)}].\] Since
  $S(k+1) + S(k-1) < S(k+4),$ by
  Lemma \ref{lemm:orbit_order}, with $k$ replaced by $k+4$, we get \[|c_{S(k+1) + S(k-1)} - c|>|c_{S(k+3)} - c|.\] On the other hand, $T^{S(k-1)}(J_{k+1}) = [c_{S(k+1)}, c_{S(k+3)}]$, then
  \[T^{S(k-1)}(c_{S(k+1)}) = c_{S(k+1) + S(k-1)} \in (c_{S(k+1)} , c_{S(k+3)})\] and by (\ref{eq:cut_time}), with $k$ replaced by $k+1$, we have that $c_{S(k+1) + S(k-1)}$ and $c_{S(k-1)}$ are in
  the same side of $c$. Since $k-1 \equiv k+3  \pmod4$, we have that $c_{S(k-1)}$ and $c_{S(k+3)}$ are on the same side of $c$, so
  $c_{S(k+1) + S(k-1)} \in (c,c_{S(k+3)})$. Thus \[ |c_{S(k+1) + S(k-1)} - c|<|c_{S(k+3)} - c|, \] a contradiction. So we must have $c_{S(k) + S(k+2)} \in (c_{S(k)}, c_{S(k+1)})$ and
  \begin{equation}
    \label{eq:empty_intersection}
    J_{k+1}~\cap~I_{k+1}~=~\emptyset. 
  \end{equation}
  This conclude the proof of the lemma.
  
\end{proof}
Taking $k' = k+1$ in Lemma (\ref{lemm:J_k}) we get $J_{k+1} \subset I_k$, and then %%%by (\ref{eq:empty_intersection}) and (\ref{eq:interval_image}) we get that%%%
\begin{equation}
  \label{eq:J_k_I_k_1}
  J_{k+1} \cup I_{k+1} \subset I_k \subseteq I_k^{S(k-1)},
\end{equation}
for every $k \geq 1$
\begin{defi}
  For $k \geq 0$ let $M_k$ be the $S(k)-$fold union \[ M_k = \bigcup_{0 \leq n < S(k-1) } I^n_k \cup \bigcup_{0 \leq n < S(k-2)} J^n_k.  \]
\end{defi}

\noindent Some examples of $M_k$,
\begin{eqnarray}
  M_0 & = & I_0 \nonumber \\
      & = & [c_1, c_2], \nonumber \\
  M_1 & = & I_1 \cup J_1 \nonumber \\
      & = & [c_2, c_5] \cup [c_4,c_1], \nonumber \\
  M_2 & = & I_2 \cup I^1_2 \cup J_2 \nonumber \\
      & = & [c_3, c_5] \cup [c_4, c_1] \cup [c_2,c_7] \nonumber \\
  M_3 & = & I_3 \cup I^1_3 \cup I^2_3 \cup J_3 \cup J^1_3 \nonumber \\
      & = & [c_{13}, c_5] \cup [c_6, c_1] \cup [c_2, c_7] \cup [c_3, c_{11}] \cup [c_4, c_{12}] \nonumber \\
  M_4 & = & I_4 \cup I^1_4 \cup I^2_4 \cup I^3_4 \cup I^4_4 \cup J_4 \cup J^1_4 \cup J^2_4 \nonumber \\
      &= & [c_{13}, c_8] \cup [c_9, c_1] \cup [c_2, c_{10}] \cup [c_3, c_{11}] \cup [c_4, c_{12}] \cup [c_{14}, c_5] \cup [c_6, c_{19}] \cup [c_{20}, c_7], \nonumber 
\end{eqnarray}
\noindent and so on.

From the definition, for every $k \geq 0$ the $S(k)$ closed intervals \[ I_k, I^1_k, \ldots , I^{S(k-1) -1}_k, J_k, J^1_k, \ldots , J^{S(k-2)-1}_k,\] are pairwise disjoint, each $M_k$
contains the set $\overline{\cO_T(c)}$ and they
form a nested sequence of closed sets $M_1 \supset M_2 \supset M_3 \supset \ldots$ with intersection equal to the cantor set $\overline{\cO_T(c)}$, for a proof of this
statements see \cite[Lemma 3.5]{LyMi93}.
Now by (\ref{eq:J_k_I_k_1}) we have that for every $1 \leq m < S(k-1)$, \[ I_{k+1}^m \cup J_{k+1}^m \subset I_k^m. \] Also by (\ref{eq:J_k_definition}) for every $0 \leq n < S(k-2)$,
\[ I_{k+1}^{S(k-1) + n} = J_k^n. \] Since the sets in $M_k$ are disjoint we get that
\begin{equation}
  \label{eq:J_k_I_k}
  \cup_{A \in M_{k+1}} (A \cap I_k) = I_{k+1} \cup J_{k+1.} 
\end{equation}

% The action of $T$ restricted to $\overline{\cO_T(c)}$ is semiconjugated to the \emph{Fibonacci shift}, this is the space $\Sigma_F \= \{ \ulx \in \{ 0, 1 \}^{\N_0}
% \colon x_{i-1}x_{i} \neq 11 \text{ for every } i \in \N \}$ together with the topology generated by the cylinders \[ [a_0a_1 \ldots a_{n-1}]_k \= \{ \ulx \in \Sigma_F \colon
%   x_{k + i} = a_i \text{ for every } 0 \leq i \leq n-1\}. \] Given $x \in \overline{\cO_T(c)}$, let $M_k(x)$ be the interval of the set $M_k$
% containing $x$, then define $a_k = 0$ if $M_k(x) = I^n_k$ and $a_k = 1$ if $M_k(x) = J^n_k$ for appropiate $n$. The sequence $\ula = a_0a_1a_2 \ldots \in \Sigma_F$ is the one that
% correspond to $x \in \overline{\cO_T(c)}$. We will denote by $M_{k}^{[a_0a_1 \ldots a_n]}$ the interval of $M_k$ corresponding to the cylinder $[a_0a_1 \ldots a_n]_0 \subset \Sigma_F$. 

\subsection{Diameter estimates}
In this section we will give an estimate on how the distances $|c_{S(k)} - c|$ decrease as $k \to \infty.$
 
\begin{lemm}
  \label{lemm:diam_estimate}
  The following limit \[ \lim_{k \to \infty} \lambda^{S(k+1)}|D_k|\] exists and is strictly positive. 
\end{lemm}

\begin{proof}
  Since \[T(D_k) = T([c,c_{S(k)}]) = I_k^1,\] by (\ref{eq:interval_image}) we have that \[T^{S(k-1)}(D_k) = [c_{S(k-1)}, c_{S(k+1)}] .\]
  By (\ref{eq:cut_time}), $c_{S(k-1)}$ and $c_{S(k+1)}$ are in opposite sides of $c$. Then $D_{k-1} \cap D_{k+1} = \{c\}$, so \[T^{S(k-1)}(D_k) = D_{k-1} \cup D_{k+1} .\] Since, by
  Lemma~\ref{lemm:T_injective}, $T^{S(k-1)}$ is injective on $I_k$ and $D_k \subset I_k$ we get that \[ |T^{S(k-1)}(D_k)| = \lambda^{S(k-1)} |D_k| = |D_{k-1}| +|D_{k+1}|.\]
  For $k \geq 0$ put  $\nu_k\= |D_k| / |D_{k+1}|$. By the above we get \[ \lambda^{S(k-1)} = \nu_{k-1} + \frac{1}{\nu_k}. \] By (\ref{def:fibonacci}),
  $\nu_k > 1$, so $0 < \nu_k^{-1} <1$. Since
  $\lambda >1$, we have $\lambda^{S(k-1)} \longrightarrow \infty$ as $k \longrightarrow \infty$. Then $\nu_k \longrightarrow \infty$ as $k \longrightarrow \infty$. So
  $\lambda^{S(k-1)} - \nu_{k-1} \longrightarrow 0$ as $k \longrightarrow \infty.$ Then, if we define $C_k \=\nu_k\lambda^{-S(k)}$, we have $0 < C_k < 1$ and $C_k \nearrow 1$
  exponentially fast as $n \longrightarrow\infty.$ By definition of $\nu_k$, we have that
  \[ \frac{|D_0|}{|D_{k+1}|} = \prod_{i=0}^k \nu_i 
                            = \prod_{i=0}^k \lambda^{S(i)}C_i 
                            =\lambda^{S(k+2) - S(1)}\prod_{i=0}^kC_i. \]
  % \begin{eqnarray}
  %   \frac{|D_0|}{|D_{k+1}|} &=& \prod_{i=0}^k \nu_i \nonumber \\
  %                           &=& \prod_{i=0}^k \lambda^{S(i)}C_i \nonumber \\
  %                           &=&\lambda^{S(k+2)}\prod_{i=0}^kC_i, \nonumber
  %   \end{eqnarray}
  Then 
  \begin{equation}
    \label{eq:diameter_estimate}
    |D_{k+1}|\lambda^{S(k+2) - S(1)} = |D_0| \left[ \prod_{i = 0}^k C_i \right]^{-1}. 
  \end{equation}
  Since $ \prod_{i = 0}^k C_i $ converge to a strictly positive number as $k \longrightarrow \infty$, the proof is complete.
\end{proof}
%%% Local Variables:
%%% mode: latex
%%% TeX-master: "Unbounded_Lyapunov_exponent"
%%% End:

\section{The invariant measure}
 
Let us denote by $\mu_P$ the unique ergodic invariant measure of $T$ restricted to $\overline{\cO_T(c)}$. As in the previous, section put
\[ S(-2) = 0, S(-1) = 1, S(0) = 1, S(1) = 2, \ldots \]   and put $\varphi \= \frac{1 + \sqrt{5}}{2}$. In this section we will
estimate the value of $\mu_P$ over the elements of $M_k$ for every $k \geq 1$.
% More specifically we have the following lemma
% \begin{lemm}
%   \label{lemm:measure_estimate}
%   For every $k \geq 1$ we have that \[ \mu_P(I_k^i) = \frac{1}{\varphi^k} \text{ \hspace{1cm} and  \hspace{1cm}} \mu_P(J_k^j) = \frac{1}{\varphi^{k+1}}. \]
%   For all $0 \leq i < S(k-1)$ and all $0 \leq j < S(k-2).$ 
% \end{lemm}

As mentioned in \cref{sec:post-critical}, we know that the set $\overline{\cO_T(c)}$ is contained in $M_k$ for very $k \geq 1$ and the sets
\begin{equation}
I_k, I_k^1, \ldots I_k^{S(k-1)-1}, J_k,J_k^1,\ldots , J_k^{S(k-2)-1}, \nonumber
\end{equation}
are disjoint. Then
\begin{equation}
  \label{eq:measure_1}
  \sum_{i=0}^{S(k-1)-1}\mu_P(I_k^i) +  \sum_{j = 0}^{S(k-2)-1}\mu_P(J_k^j) = 1.
\end{equation}
Since $T$ restricted to $\overline{\cO_T(c)}$  is injective, except at the critical point that has two preimages, we have that
\begin{eqnarray}
  \label{eq:measure_2}
  \mu_P(I_k^i) &=& \mu_P(I_k^j)\\
  \mu_P(J_k^p) &=& \mu_P(J_k^q),
\end{eqnarray}
for every $0 \leq i,j < S(k-1)$ and $0 \leq p,q < S(k-2)$. Then we can write (\ref{eq:measure_1}) as
\begin{equation}
  \label{eq:measure_3}
  S(k-1)\mu_P(I_k) + S(k-2)\mu_P(J_k) = 1
\end{equation}
Since \[ I_k \sqcup J_k \subset I_{k-1},\] we have that
\begin{equation}
  \label{eq:measure_4}
  \mu_P(I_k) + \mu_P(J_k) = \mu_P(I_{k-1}).
\end{equation}
And since \[J_{k-1} =I_k^{S(k-1)},\]using (\ref{eq:measure_2}) with $k$ replaced by $k=1$, we have that
\begin{equation}
  \label{eq:measure_5}
  \mu_P(J_{k-1}) = \mu_P(I_k).
\end{equation}
Combining (\ref{eq:measure_4}) and (\ref{eq:measure_5}) we can write
\begin{equation}
  \label{eq:matrix_measure}
  \begin{bmatrix}
    1 & 1 \\ 1 & 0
  \end{bmatrix}
  \begin{bmatrix}
    \mu_P(I_k) \\
    \mu_P(J_k)
  \end{bmatrix}
  =
  \begin{bmatrix}
    \mu_P(I_{k-1}) \\ \mu_P(J_{k-1})
  \end{bmatrix}.  
\end{equation}
\begin{lemm}
  \label{lemm:measure_estimate}
  For every $m \geq 1$ we have \[ \mu_P(I_m) = \frac{1}{\varphi^m} \hspace{0.5cm} \text{ and } \hspace{0.5cm} \mu_P(J_m) = \frac{1}{\varphi^{m+1}}. \] 
\end{lemm}
\begin{proof}
  We will use induction to prove the lemma. For $m = 1$.
  We can apply  $k-2$ times the equation (\ref{eq:matrix_measure}) to
  \[
    \begin{bmatrix}
      \mu_P(I_{k-1}) \\ \mu_P(J_{k-1})
    \end{bmatrix},
  \]
  and we can write (\ref{eq:matrix_measure}) as
  \begin{equation}
    \label{eq:matrix_measure_2}
    \begin{bmatrix}
      1 & 1 \\ 1 & 0
    \end{bmatrix}^{k-1}
    \begin{bmatrix}
      \mu_P(I_k) \\
      \mu_P(J_k)
    \end{bmatrix}
    =
    \begin{bmatrix}
      \mu_P(I_{1}) \\ \mu_P(J_{1})
    \end{bmatrix}.  
  \end{equation}
  Using that
  
  \[
    \begin{bmatrix}
      1 & 1 \\ 1& 0
    \end{bmatrix}^{k-1}
    =
    \begin{bmatrix}
      S(k-2) & S(k-3) \\ S(k-3) & S(k-4)
    \end{bmatrix},
  \] for $k \geq 2$. We can write
  \begin{equation}
    \label{eq:matrix_measure_3}
    \begin{bmatrix}
      \mu_P(I_1) \\
      \mu_P(J_1)
    \end{bmatrix}
    =
    \begin{bmatrix}
      S(k-2)\mu_P(I_k) + S(k-3)\mu_P(J_k)  \\ S(k-3)\mu_P(I_k) + S(k-4)\mu_P(J_k)
    \end{bmatrix}.  
  \end{equation}
  Multiplying the first equation in (\ref{eq:matrix_measure_3}) by $\frac{S(k)}{S(k-1)}$ we get
  \begin{equation}
    \label{eq:measure_6}
    \frac{S(k)}{S(k-1)} \mu_P(I_1) = \frac{S(k-2)}{S(k-1)}S(k)\mu_P(I_k) + \frac{S(k-3)}{S(k-2)}\frac{S(k)}{S(k-1)}S(k-2)\mu_P(J_k).
  \end{equation}
  Using (\ref{eq:measure_3}) we can write (\ref{eq:measure_6}) as
  \small
  \begin{equation}
    \label{eq:measure_7}
    \frac{S(k)}{S(k-1)} \mu_P(I_1) = S(k)\mu_P(I_k) \left[ \frac{S(k-2)}{S(k-1)} - \frac{S(k-3)}{S(k-2)} \right] + \frac{S(k-3)}{S(k-2)} \frac{S(k)}{S(k-1)}.
  \end{equation}
  \normalsize
  Since $\frac{S(k)}{S(k-1)} \longrightarrow \varphi$ as $k \longrightarrow \infty$, taking limit on (\ref{eq:measure_7}) over $k$ we get \[ \varphi \mu_P(I_1) = 1,  \] and
  then \[\mu_P(I_1) =\frac{1}{\varphi}. \] Using (\ref{eq:measure_3}) with $k$ replaced by $1$ we get that \[ \mu_P(J_1) = \frac{1}{\varphi^2}. \]
  So the lemma holds for $m =1$.

  Suppose now that the result is true for $m$. By (\ref{eq:measure_5}) we have that
  \[ \mu_P(I_{m+1})= \mu_P(J_m) = \frac{1}{\varphi^{m+1}}. \] By (\ref{eq:measure_4}) we have that
  \begin{eqnarray}
    \mu_P(J_{m+1}) &=& \mu_P(I_m) - \mu_P(I_{m+1})  \nonumber \\
                   &=& \frac{1}{\varphi^m} - \frac{1}{\varphi^{m+1}} \nonumber \\
                   &=& \frac{1}{\varphi^m} \left( 1 - \frac{1}{\varphi} \right) \nonumber \\
                   &=& \frac{1}{\varphi^m} \frac{1}{\varphi^2} \nonumber \\
                   &=& \frac{1}{\varphi^{m+2}}, \nonumber
  \end{eqnarray}
 and we get the result.
\end{proof}
%%% Local Variables:
%%% mode: latex
%%% TeX-master: "Unbounded_Lyapunov_exponent"
%%% End:

\section{Proof of Proposition \ref{prop:0}}
In this section we will give the proof of Proposition \ref{prop:0}. We use the same notation as in the previous section. Let $h \colon [-1,1] \to [-1,1]$ be as in
Proposition \ref{prop:0} and $f = h \circ T \circ h^{-1}$. As in Theorem \ref{thm:1}, put 
$\tilde{c} \= h(0)$.

Since $h$ has a non-flat critical point at $0$, by \eqref{eq:right_crit_point} and \eqref{eq:left_crit_point} there are $\alpha^+ >0$, $\alpha^- >0$, and $\delta > 0$ such that
  \begin{equation}
    \label{eq:l_1}
    e^{-M}|\hat{x}|^{\alpha^+} \leq |h'(\hat{x})| \leq e^{M} |\hat{x}|^{\alpha^+},
  \end{equation}
  for every $\hat{x} \in (0,\delta)$ and

  \begin{equation}
    \label{eq:l_2}
    e^{-M}|\hat{x}|^{\alpha^-} \leq |h'(\hat{x})| \leq e^{M} |\hat{x}|^{\alpha^-},
  \end{equation}
  for every $\hat{x} \in (-\delta, 0)$.
% First, take $\delta > 0$ small enough so that $(-\delta,\delta) \subset I_k$ for some $k >2$.
% Since $h$ has a unique critical point at 
Since $c = 0$ and $c \notin I_k^1 = [c_{c_{S(k)+1}}, c_1],$ we have that there exist positive real numbers $W_1$ and $W_2$ such that for every
$x \in I_k$
\begin{equation}
  \label{eq:seq_bound}
  W_1 \leq |h'(T(x))| \leq W_2.
\end{equation}
\begin{proof}[Proof of Proposition \ref{prop:0}]
  
  By the chain rule we have
  \begin{equation}
    \label{eq:l_3}
    f'(x) = \lambda \frac{h'(T(h^{-1}(x)))}{h'(h^{-1}(x))},
  \end{equation}
  for every $x \in (h(-\delta), h(\delta)) \setminus \{ \tilde{c} \}.$ Let 
  $K \= \max \{ \lambda^{-1} e^M W_1^{-1}, \lambda e^M W_2 \}$. Then by \eqref{eq:seq_bound},
  \eqref{eq:l_1},
  \eqref{eq:l_2}, and \eqref{eq:l_3} we have that for every $x \in (\tilde{c}, h(\delta))$
  \begin{equation}
    \label{eq:l_6}
    \frac{1}{K|h^{-1}(x)|^{\alpha^+}} \leq |f'(x)| \leq \frac{K}{|h^{-1}(x)|^{\alpha^+}}, 
  \end{equation}
  and for every $x \in (h(-\delta), \tilde{c})$
  \begin{equation}
    \label{eq:l_7}
    \frac{1}{K|h^{-1}(x)|^{\alpha^-}} \leq |f'(x)| \leq \frac{K}{|h^{-1}(x)|^{\alpha^-}}.
  \end{equation}
 Now, from \eqref{eq:l_1} and \eqref{eq:l_2} there exist $M_1 >0$ and $M_2 > 0$ such that for every $x \in (0,\delta)
  $ \[ M_1^{-1} |x|^{\alpha^+ + 1} \leq |h(x)| \leq M_1 |x|^{\alpha^+ + 1}, \]
  and for every $x \in (-\delta,0)$
  \[ M_2^{-1} |x|^{\alpha^- + 1} \leq |h(x)| \leq M_2 |x|^{\alpha^- + 1}.  \]
  Since $h$ is a homeomorphism, there exist constants $M_3 > 0$ and $M_4 > 0$ such that  for every $x \in (\tilde{c}, h(\delta))$
  \begin{equation}
    \label{eq:l_4}
    M_3^{-1} |x - \tilde{c}|^{\frac{1}{\alpha^+ +1}} \leq |h^{-1}(x)| \leq 
    M_3 |x - \tilde{c}|^{\frac{1}{\alpha^+ + 1}}, 
  \end{equation}
  and for every $x \in (h(-\delta), \tilde{c})$
  \begin{equation}
    \label{eq:l_5}
    M_4^{-1} |x - \tilde{c}|^{\frac{1}{\alpha^- +1}} \leq |h^{-1}(x)| \leq 
    M_4 |x - \tilde{c}|^{\frac{1}{\alpha^- + 1}}.
  \end{equation}
  Then, by \eqref{eq:l_6}, \eqref{eq:l_7}, \eqref{eq:l_4}, and \eqref{eq:l_5}  we have that for every $x \in (\tilde{c}, h(\delta))$
  \[ \frac{1}{M_3 |x - \tilde{c}|^{\frac{\alpha^+}{\alpha^+ +1}}}  \leq 
  |f'(x)| \frac{M_3}{|x - \tilde{c}|^{\frac{\alpha^+}{\alpha^+ +1}}} , \]
  and for every $x \in (\tilde{c}, h(-\delta))$
  \[ \frac{1}{M_4 |x - \tilde{c}|^{\frac{\alpha^-}{\alpha^- +1}}} \leq 
  |f'(x)| \leq \frac{M_4}{|x - \tilde{c}|^{\frac{\alpha^-}{\alpha^- +1}}}. \]
  Thus, $f$ has a Lorenz-like singularity at $\tilde{c}$.
\end{proof}

%%% Local Variables:
%%% mode: latex
%%% TeX-master: "Unbounded_Lyapunov_exponent"
%%% End:

\section{Proof of Proposition \ref{prop:1}}
In this section, we will give the proof of Proposition \ref{prop:1}. We will use the same notation as in the previous sections.  

First, take $\alpha \= \max \{ \alpha^+, \alpha ^- \} $, where $\alpha^+$ and $\alpha^-$ are given in Proposition \ref{prop:0}. From  \eqref{eq:thm_1} and
\eqref{eq:thm_2}, we get that for every $x \in (h(-\delta), h(\delta)) \setminus \{ \tilde{c} \}$

\begin{equation}
  \label{eq:thm_condition}
  \frac{1}{K |h^{-1}(x)|^{\alpha}} \leq |f'(x)|.
\end{equation}

\begin{proof}[Proof of Proposition \ref{prop:1}]
  First we prove $(i)$. By Lemma \ref{lemm:diam_estimate}, there is $k\geq 2$ so that (\ref{eq:thm_condition}) holds on
  $h(I_k) \setminus \{ \tilde{c} \}$ and such that $|f'| >1$ on
  $h(I_k) \setminus \{ \tilde{c} \}$. For $n > k$ put $L_n \= \lambda^{S(n+1)}|D_n|.$ Define
  \[ \log|f'|^+ \= \max\{0,\log|f'|  \}\hspace{15pt} \text{ and } \hspace{15pt} \log|f'|^- \= \{0,-\log|f'|  \}, \] on $I \setminus \{ \tilde{c} \}$.
  By Lemma \ref{lemm:J_k}, for every $n > k$, we have $J_n \subset I_k$ and for every 
  $k < n < n'$, we have $J_n \cap  J_{n'} = \emptyset$. So, since $\tilde{\mu}_P(\{ \tilde{c} \}) = 0$
  \[ \int \log|f'|^+d\tilde{\mu}_P \geq \int_{h(I_k)}\log|f'|d\tilde{\mu}_P\geq \sum_{n > k}\int_{h(J_n)}\log|f'|d\tilde{\mu}_P.\]
  Recall that $\varphi \= \frac{1 + \sqrt{5}}{2}$. Then  for each $n > k$ and $x \in J_n$, we have by Lemma (\ref{lemm:J_k}) and (\ref{eq:thm_condition})
  \begin{equation}
    \label{eq:f_derivative}
    |f'(h(x))| \geq  K^{-1} \frac{1}{|D_{n-1}|^{\alpha}} = K^{-1}\lambda^{\alpha S(n)} L_{n-1}^{-\alpha}. 
  \end{equation}
  %Let $M_n \= \log \left(\lambda K^{-1}\left[ \frac{1}{|D_0|} \prod_{i=0}^{n-2}C_i \right]^{\alpha}\right) $.
  By the above together with Lemma \ref{lemm:measure_estimate} and the fact that $S(n) \geq \frac{1}{3}\varphi^{n+2}$
  % \begin{eqnarray}
  %   \int_{h(J_n)} \log|f'|d\tilde{\mu}_P & \geq & \left( \frac{1}{\varphi} \right)^n \log{\left(\lambda K^{-1} \left( \frac{1}{|D_0|} \lambda^{S{(n)}} \prod_{i=0}^{n-2}C_i \right)^{\alpha}
  %                                                 \right)} \nonumber \\
  %                                        & = & \left( \frac{1}{\varphi} \right)^n \left[ \alpha S(n) \log(\lambda) + \log{\left(\lambda K^{-1}\left[ \frac{1}{|D_0|}
  %                                              \prod_{i=0}^{n-2}C_i \right]^{\alpha} \right)}  \right] \nonumber \\
  %                                        & = & \left( \frac{1}{\varphi} \right)^n \alpha \frac{1}{\sqrt{5}} \left( \varphi^n + \left( \frac{-1}{\varphi}  \right)^n  \right)
  %                                              \log(\lambda) + \left( \frac{1}{\varphi} \right)^n  M_n},  \nonumber \\
  %                                       & = & \frac{\alpha}{\sqrt{5}} \log(\lambda) + \alpha \left( \frac{-1}{\varphi^2} \right)^n\log{\lambda} + \left( \frac{1}{\varphi} \right)^n
  %                                              M_n, \nonumber
                                                %     \end{eqnarray}
  
  \begin{eqnarray}
  \label{eq:non_integrable_bound}
    \int_{h(J_n)} \log|f'|d\tilde{\mu}_P & \geq & \tilde{\mu}_P(h(J_n)) \log \left| K^{-1} \lambda^{\alpha S(n)} L_{n-1}^{-\alpha}\right|  \\
                                         %& \geq & \left( \frac{1}{\varphi} \right)^n \log \left( \lambda K^{-1} \left( \frac{1}{|D_0|} \lambda^{S(n)} \prod_{i=0}^{n-2} C_i
                                                  %\right)^{\alpha}\right) \nonumber \\
                                         & \geq & \left( \frac{1}{\varphi} \right)^{n+1} \left[ \alpha S(n) \log(\lambda) + \alpha \log(K^{1/\alpha}L_{n-1})^{-1}  \right] \nonumber \\
                                         & \geq & \frac{\varphi \alpha}{3}\log(\lambda) + \alpha \left( \frac{1}{\varphi} \right)^{n+1} \log(K^{1/\alpha}L_{n-1})^{-1}.  \nonumber
  \end{eqnarray}
  By Lemma \ref{lemm:diam_estimate}, $\left( \frac{1}{\varphi} \right)^{n+1} \log(K^{1/\alpha}L_{n-1})^{-1} \longrightarrow 0$ as $n \longrightarrow \infty$. We get that
  %Finally we get \[ \int \log|f'|^+d\tilde{\mu}_P \geq \sum_{n >k} \left[ \frac{\alpha}{3} \log(\lambda) + \alpha \left( \frac{1}{\varphi} \right)^n \log(KL_n)^{-1} \right].\]
                                                  %                                                   Since $\left( \frac{1}{\varphi} \right)^n M_n \to 0,$  we get that
  \begin{equation}
    \label{eq:positive_unbounded}
    \int \log|f'|^+d\tilde{\mu}_P = +\infty.  
  \end{equation}
  
  Now we prove $(ii)$. Suppose by contradiction that 

  \begin{equation}
    \label{eq:finite_neg_int}
    \int \log|f'|^-d\tilde{\mu}_P < +\infty.
  \end{equation}
  By the chain rule
  
  \begin{equation}
    \label{eq:log_relation}
    \log|f'(x)| = \log(\lambda) + \log|h'(T(h^{-1}(x)))| - \log|h'(h^{-1}(x))|,
  \end{equation}
  on $I \setminus \{\tilde{c}\}$
  Since $h'$ has a unique critical point, $\log|h'|$ is bounded away from the
  critical point. In particular, is bounded from above in all $I$. Then $-\log|h'|$ is bounded from below on $I$. In particular, since $\tilde{\mu}_P(\{\tilde{c}\})=0$, the integral
  \[ \int \log|h' \circ h^{-1}| d\tilde{\mu}_P,  \] is defined. Since the only critical points of $f$ are the points in
  $f^{-1}(\{ \tilde{c}\})$, we have that $\log|f'|$ is bounded away from $\{\tilde{c}\} \cup f^{-1}\{\tilde{c} \}$. Let $\tV$ $\subset I \setminus \{\tilde{c}\}$ be a neighborhood of
  $f^{-1}\{ \tilde{c} \}$ such that $\log|f'(x)| < 0$ for $x \in \tV$, then by (\ref{eq:finite_neg_int}) and (\ref{eq:log_relation})
  \[ -\infty < \int_{\tilde{V}}\log|f'|d\tilde{\mu}_P = \log(\lambda)d\tilde{\mu}_P(\tilde{V}) + \int_{\tilde{V}}\left(\log|h'\circ T \circ h^{-1}| - \log|h' \circ h^{-1}|\right) 
    d\tilde{\mu}_P. \]  Since $h^{-1}(\tV)$ is a neighborhood of $T^{-1}(c)$, the function $-\log|h'\circ h^{-1}|$ is bounded on $\tV$. On the other hand, since
  \[ h^{-1} \circ f(x) = T \circ h^{-1}(x) = c,\] we have  $h' \circ T \circ h^{-1}(x) = 0$ if $x \in f^{-1}(\tilde{c})$. Thus
  $h' \circ T \circ h^{-1}(x) \neq 0$ for $x \in I \setminus \tV$. Then $\log|h' \circ T \circ h^{-1}|$ is bounded in $I \setminus \tV $. So
  \begin{equation}
    \label{eq:int_bound}
    \int_{I \setminus \tV}\log|h' \circ T\circ h^{-1}| d\tilde{\mu}_P > -\infty. 
  \end{equation}
  Now,
  \begin{eqnarray}
    -\infty & < & \int_{\tilde{V}}\log|f'|d\tilde{\mu}_P \nonumber \\
            & \leq & \left( \log(\lambda)+ \max_{x \in \tilde{V}}\{-\log|h' \circ h^{-1}(x)|\}\right)\tilde{\mu}_P(\tilde{V}) + 
                     \int_{\tilde{V}}\left(\log|h'\circ T \circ h^{-1}| \right) d\tilde{\mu}_P. \nonumber 
  \end{eqnarray}
  So
  \begin{equation}
    \label{eq:other_inequality}
    \int_{\tilde{V}}\log|h' \circ T \circ h^{-1}| d\tilde{\mu}_P > - \infty.
  \end{equation}
  Together with (\ref{eq:int_bound}) this implies that \[\int\log|h' \circ T \circ h^{-1}| d\tilde{\mu}_P,\] is finite. Since the integral
  \[ \int -\log|h' \circ h^{-1}| d\tilde{\mu}_P, \] is defined, we have 
  
  \begin{eqnarray}
    \int \log|f'|d\tilde{\mu}_P & = & \log(\lambda) + \int \log|h'\circ T \circ h^{-1}|  d\tilde{\mu}_P  + \int- \log|h'\circ  h^{-1}| d\tilde{\mu}_P\nonumber \\
                                & = & \log(\lambda) + \int \log|h' \circ h^{-1} \circ f|  d\tilde{\mu}_P  + \int- \log|h'\circ  h^{-1}| d\tilde{\mu}_P,  \nonumber 
  \end{eqnarray}
  and since $\tilde{\mu}_P$ is $f$ invariant we get \[  \int \log|f'|d\tilde{\mu}_P  = \log(\lambda), \] contradicting (\ref{eq:positive_unbounded}). This contradiction completes
 the proof of part $(ii)$.
 
 Finally we prove $(iii)$. By Proposition \ref{prop:0}, $f$ has a Lorenz-like singularity at~$\tilde{c}$, then there 
 exist $\delta >0$, $\ell^+ >0,$ $\ell^- >0$ and $L>0$ such that \eqref{eq:right_lo-like} holds for every 
 $x \in (\tilde{c}, \tilde{c} + \delta)$ and \eqref{eq:left_lo-like} holds for every 
 $x \in (\tilde{c} - \delta, \tilde{c})$. Let $\ell \= \max \{ \ell^+,\ell^- \},$ and choose 
 $0 < \hat{\delta} \leq \delta$, so that $\log|f'(x)| > 0$ and $\dist(x,\cS(f)) = |x - \tilde{c}|$ for every 
 $x \in (\tilde{c} - \hat{\delta}, \tilde{c} + \hat{\delta}) \setminus \{ \tilde{c} \}$. Let $m \geq 2$ be so that
 $I_m \subset (\tilde{c} - \hat{\delta}, \tilde{c} + \hat{\delta})$. Then, for every 
 $x \in I_m \setminus \{ \tilde{c} \}$ we have \[ \log|f'(x)| \leq \log(L) - \ell \log (\dist(x,\cS(f))).\] So, for every
 $n\geq m$ \[ \int_{h(J_n)} \log|f'| d\tilde{\mu}_P \leq \log(L)\tilde{\mu}_P(J_n) + 
 \ell \int_{h(J_n)} |\log|x -\tilde{c}|| d\tilde{\mu}_P(x). \] By \eqref{eq:non_integrable_bound} and Lemma
 \ref{lemm:diam_estimate} we get that
 \[ +\infty = \int_{h(I_m)} |\log(\dist(x, \cS(f)))| d\tilde{\mu}_P(x), \] so 
 $\log(\dist(x,\cS(f))) \notin L^1(\tilde{\mu}_P)$. This conclude the proof of the proposition.
\end{proof}
%%% Local Variables:
%%% mode: latex
%%% TeX-master: "Unbounded_Lyapunov_exponent"
%%% End:

\section{Proof of Proposition \ref{prop:2}}
In this section, we will prove Proposition \ref{prop:2}. 
% thus we will prove that for every $x \in \overline{\cO_f(\tilde{c})}$ with $\tilde{c} \notin \cO_f(x)$,
% the sequence
% \begin{equation}
%   \label{eq:pw_Lyap_seq}
%   \left( \frac{1}{n}\log|(f^n)^n(x)| \right)_{n \geq 1}, 
% \end{equation}
% does not have a limit. So the Lyapunov exponent of $f$ at $x$ does not exist. 
We will use the same notation as in the previous sections.
Recall that $f'$ is not defined at $\tilde{c}$, so for $x$ in $I$ whose orbit contains $\tilde{c}$
the derivative $(f^n)'$ at $x$ does not exist for large $n$.
% We considere $h \colon I \longrightarrow I$ as in Theorem \ref{thm:1}, thus there exist $\alpha^+ >0$,
% $\alpha^- > 0$, $M>0,$ and $\delta > 0$ such that
% \begin{equation}
%     \label{eq:right_crit_point_h}
%     e^{-M}|\hat{x}|^{\alpha^+} \leq |h'(\hat{x})| \leq e^{M} |\hat{x}|^{\alpha^+},
%   \end{equation}
%   for every $\hat{x} \in (0,\delta)$ and

%   \begin{equation}
%     \label{eq:left_crit_point_h}
%     e^{-M}|\hat{x}|^{\alpha^-} \leq |h'(\hat{x})| \leq e^{M} |\hat{x}|^{\alpha^-},
%   \end{equation}
% for every $\hat{x} \in (-\delta, 0)$. 
Let $\alpha^+$ and $\alpha^-$ be the right and left critical orders of $0$ as the critical point of
$h$, and let 
$\alpha \= \max \{\alpha^+, \alpha^- \}$. Let $M >0$ be as in \eqref{eq:l_1} and \eqref{eq:l_2}.
Fix $k >2$ big enough so that (\ref{eq:thm_condition}) holds on $h(I_k)$, and
\eqref{eq:l_1}, and \eqref{eq:l_2} holds on $I_k$.
 
% \[ \limsup_{n \to \infty} \frac{1}{l(x) + n} \log|(f^{l(x) + n})'(x)|  =  \limsup_{n \to \infty} \frac{1}{n}
% \log|(f^{n})'(x)|, \] and
% \[ \liminf_{n \to \infty} \frac{1}{l(x) + n} \log|(f^{l(x) + n})'(x)|  =  \liminf_{n \to \infty} \frac{1}{n}
% \log|(f^{n})'(x)|. \]
% Then is enough to prove that \eqref{eq:prop_2} holds for every 
% $x \in h(I_k) \cap \overline{\cO_f(\tilde{c})}$. 
% to show that (\ref{eq:pw_Lyap_seq}) does not have a limit it is enough to prove that for every $x \in
% h(I_k)$
% \[ \liminf_{n \to \infty} \frac{1}{n} \log|(f^{n})'(x)|  <  \limsup_{n \to \infty} \frac{1}{n} \log|(f^{n})'(x)|. \]

For every $x \in I$ such that $\tilde{c} \notin \cO_f(x)$ we put
\begin{equation}
  \label{eq:up_Lyap_exp}
  \chi_f^+(x) \= \limsup_{n \to \infty} \frac{1}{n} \log|(f^n)'(x)|,
\end{equation}
and 
\begin{equation}
  \label{eq:lo_Lyap_exp}
  \chi_f^-(x) \= \liminf_{n \to \infty} \frac{1}{n} \log|(f^n)'(x)|.
\end{equation}
%Then the sequence (\ref{eq:pw_Lyap_seq}) converges if and only in \[ \chi_f^-(x) = \chi_f^+(x).  \]

% \begin{equation}
%   \label{eq:pw_lyap_2}
%   \chi_f(x) = \log\lambda + \lim_{n \to \infty} \frac{1}{n} \log \left|h'(T^n(h^{-1}(x)))\right|.
% \end{equation}
% Now if $x \in \cO_f(\tilde{c})$, then $h^{-1}(x)$ belong to some of the \[ I_1, \ldots, I_{S(k-1)}, J_1, \ldots , J_{S(k-2)}. \] 
% \begin{eqnarray}
%   \chi_f(x) &=& \log\lambda + \lim_{n \to \infty} \frac{1}{n} \log \left|h'(T^n(h^{-1}(x)))\right| \nonumber \\
%             &=& \log\lambda + \lim_{n \to \infty} \frac{1}{l(x) + n} \log \left|h'(T^n(T^{l(x)}(h^{-1}(x))))\right|. \nonumber \\
%             &=& \log\lambda + \lim_{n \to \infty} \frac{1}{n} \log \left|h'(T^n(h^{-1}(f^{l(x)}(x))))\right|
% \end{eqnarray}
% If we put $\tilde{x} \= f^{l(x)}(x)$ we get
% \[ \chi_f(x) = \log\lambda + \lim_{n \to \infty} \frac{1}{n} \log \left|h'(T^n(h^{-1}(\tilde{x})))\right|.  \]
For any $x \in h(I_k)$ we put $\hat{x} \= h^{-1}(x) \in I_k.$ The proof of the Proposition 
\ref{prop:2} is given after the following lemma.
% The following lemma will give us an upper bound for the return time for $I_{k+i}$.
\begin{lemm}
  \label{lemm:time_bound}
  For every $\hat{x} \in I_k \cap \overline{\cO_T(c)}$ there exists an increasing sequence of positive
  integers $\{n_i\}_{i \geq 1}$ such that 
  \[T^{n_i}(\hat{x}) \in I_{k+i} \hspace{1cm} \text{ and } \hspace{1cm} T^m(\hat{x}) \notin I_{k+i+1},\]
  for all $i \geq 1$ and all $n_i +1 \leq m < n_{i+1}$. Moreover, 
  \begin{equation}
    \label{eq:return_time_bound}
    S(k+i) - S(k) \leq n_i \leq S(k+i+2) - S(k+2) ,
  \end{equation}
    for all $i > 1$.
\end{lemm}

\begin{proof} We will prove the lemma by induction.
  Let $\hat{x} \in I_k \cap \overline{\cO_T(c)}$. Recall that for any integer $k' \geq 1$, we have that \[ I_{k'}, I^1_{k'}, \ldots I_{k'}^{S(k'-1)-1},J_{k'}, J_{k'}^1, \ldots , J_{k'}^{S(k'-2)-1},  \]
  are pairwise disjoint. Now, by (\ref{eq:J_k_I_k}) \[ \hat{x} \in I_{k+1} \hspace{1cm}\text{ or } \hspace{1cm} \hat{x} \in J_{k+1}.\]
  If  $\hat{x} \in J_{k+1},$ for every $1 \leq m <S(k-1)$ \[ T^m(\hat{x}) \in J_{k+1}^m, \] thus \[ T^m(\hat{x}) \notin I_{k+1} \] and   \[ T^{S(k-1)}(\hat{x}) \in I_{k+1}. \]
  In this case $n_1 = S(k-1)$ satisfies the desired properties. 
  If $\hat{x} \in I_{k+1}, $ for every $1 \leq m < S(k)$ \[T^m(\hat{x}) \in I_{k+1}^m ,\] thus \[T^m(\hat{x}) \notin I_{k+1}, \]
  and  \[ T^{S(k)}(\hat{x}) \in I_{k+1} \hspace{1cm}\text{ or } \hspace{1cm} T^{S(k)}(\hat{x}) \in J_{k+1}.\] In the former case $n_1=S(k)$ satisfies the desired properties. In the later
  case we have that for $1 \leq m < S(k) + S(k-1) = S(k+1)$ \[ T^m(\hat{x}) \notin I_{k+1}\] and \[ T^{S(k+1)}(\hat{x}) \in I_{k+1}.\] So $n_1= S(k+1)$ satisfies the desired properties.
  So we have \[S(k-1) \leq n_1 \leq S(k+1). \]
  Now suppose that for some $i \geq 1$ there is $n_i$ satisfying the conclusions of the lemma. Thus \[ T^{n_i}(\hat{x}) \in I_{k+i}\] and \[S(k+i)-S(k) \leq n_i \leq S(k+i+2)-S(k+2).\]
  By (\ref{eq:J_k_I_k}) \[ T^{n_i}(\hat{x}) \in I_{k+i+1} \hspace{1cm}\text{ or } \hspace{1cm} T^{n_i}(\hat{x}) \in J_{k+i+1}.\]
  If  $T^{n_i}(\hat{x}) \in J_{k+i+1},$ for every $1 \leq m <S(k+i-1)$ \[ T^{m+n_i}(\hat{x}) \in J_{k+i+1}^m, \] thus \[ T^{m+n_i}(\hat{x}) \notin I_{k+i+1} \] and
  \[ T^{S(k+i-1) +n_i}(\hat{x}) \in I_{k+i+1}. \] In this case $n_{i+1} = S(k+i-1) + n_i$ satisfies the desired properties. 
  If $T^{n_i}(\hat{x}) \in I_{k+i+1}, $ for every $1 \leq m < S(k+i)$ \[T^{m+n_i}(\hat{x}) \in I_{k+i+1}^m ,\] thus \[T^{m+n_i}(\hat{x}) \notin I_{k+i+1}, \]
  and  \[ T^{S(k+i) +n_i}(\hat{x}) \in I_{k+i+1} \hspace{1cm}\text{ or } \hspace{1cm} T^{S(k+i)+n_i}(\hat{x}) \in J_{k+i+1}.\] In the former case $n_{i+1}= S(k+i) + n_i$ satisfies the
  desired properties. In the later case we have that for $1 \leq m < S(k+i) + S(k+i-1) = S(k+i+1)$ \[ T^{m+n_i}(\hat{x}) \notin I_{k+i+1} \] and
  \[ T^{S(k+i+1) +n_i}(\hat{x}) \in I_{k+i+1}.\] So $n_{i+1}= S(k+i+1) + n_i$ satisfies the desired properties.
  So we have \[S(k+i-1) \leq n_{n_i+1} \leq S(k+i+1) + n_i. \] Since $n_i$ satisfies (\ref{eq:return_time_bound})
  \[S(k+i+1) - S(k) \leq n_{i+1} \leq S(k+i+3) - S(k+2).  \] This conclude the proof of the lemma.  
\end{proof}

\begin{proof}[Proof of Proposition \ref{prop:2}]
  Let $x \in \overline{\cO_f(\tilde{c})}$, with $\tilde{c} \notin \cO_f(x)$. Then
  $\hat{x} = h^{-1}(x) \in \overline{\cO_T(c)}$, and $ c \notin \cO_T(\hat{x})$. 
  By the chain rule, we have that for very $n \geq 1$
\begin{equation}
  \label{eq:n_derivative}
  (f^n)'(x) = \prod_{i=0}^{n-1} \lambda \frac{h'(T(T^i(\hat{x})))}{h'(T^i(\hat{x}))} = 
  \lambda^n \frac{h'(T^n(\hat{x}))}{h'(\hat{x})}. 
\end{equation}
Then,
\begin{equation}
  \label{eq:n_derivative_1}
  \frac{1}{n}\log|(f^n)'(x)| = \log \lambda + \frac{1}{n} \log |h'(T^n(\hat{x}))| - 
  \frac{1}{n} \log |h'(\hat{x})|,
\end{equation}
for every $n \geq 1$. So, using \eqref{eq:up_Lyap_exp} and \eqref{eq:lo_Lyap_exp}, we get
\begin{equation}
  \label{eq:up_Lyap_exp_1}
  \chi_f^+(x)  = \log \lambda + \limsup_{n \to \infty} \frac{1}{n} \log|h'(T^n(\hat{x}))|,
\end{equation}
and 
\begin{equation}
  \label{eq:lo_Lyap_exp_1}
  \chi_f^-(x)  = \log \lambda +  \liminf_{n \to \infty} \frac{1}{n} \log|h'(T^n(\hat{x}))|.
\end{equation}

  Now, since $\hat{x} \in \overline{\cO_T(c)}$, we have that $\hat{x}$ belongs to one of the following
  sets \[ I_k,I_k^1, \ldots, I_k^{S(k-1)-1}, J_k, \ldots , J_k^{S(k-2)-1}.\]
  So there exists $0 \leq l_k(x) < S(k)$ such that $T^{l_k(x)}(\hat{x})  \in I_k$.
  Let $\{n_i\}_{i \geq 1}$ be as in Lemma~\ref{lemm:time_bound}, for $T^{l_k(x)}(x)$. 
  Note that for every $i \geq 1$
  \[T^{n_i+l_k(x) +1}(\hat{x}) \in I_{k+i}^1 \subset [c_{S(k)+1}, c_1]. \] Then by
  (\ref{eq:seq_bound}) and (\ref{eq:up_Lyap_exp_1}) we have that 
  \[ \log \lambda \leq \chi_f^+(x). \]

Now by (\ref{eq:return_time_bound}) we have that
\begin{equation}
  \label{eq:rho_bounds}
  \frac{1}{S(k+i+2)} \leq \frac{1}{n_i}.
\end{equation}
Also, since for every $i \geq 1$ \[ T^{n_i+l_k(x)}(\hat{x}) \in I_{k+i},\] by (\ref{def:fibonacci})
and (\ref{eq:I_K}), we get
\begin{equation}
  \label{eq:bound_1}
  |T^{n_i+l_k(x)}(\hat{x})| \leq |c_{S(k+i)} | = |D_{k+i}|.
\end{equation}
By (\ref{eq:thm_condition}) we have 
\begin{equation}
  \label{eq:derivative_bound}
  \frac{1}{\lambda K|T^{n_i+l_k(x)}(\hat{x})|^{\alpha}} \leq 
  \frac{|h'(T^{n_i+l_k(x)+1}(\hat{x}))|}{|h'(T^{n_i+l_k(x)}(\hat{x}))|}. 
\end{equation}
Combining (\ref{eq:seq_bound}), (\ref{eq:bound_1}) and (\ref{eq:derivative_bound}) we get
\begin{equation}
  \label{eq:bound_2}
  |h'(T^{n_i+l_k(x)}(\hat{x}))| \leq \lambda K W_2 |D_{k+i}|^{\alpha}.
\end{equation}
Since $|D_{k+i}| \longrightarrow 0$ as $i \longrightarrow \infty$, there exists $i' \geq 1$ such that for every $i \geq i'$
\[|D_{k+i}| < \left( \frac{1}{\lambda K W_2} \right)^{1/\alpha}. \]
Then for every $i \geq i'$ from (\ref{eq:bound_2}) we get
\[\log |h'(T^{n_i+l_k(x)}(\hat{x}))| \leq \log \left( \lambda KW_2|D_{k+i}|^{\alpha} \right) < 0.  \]
By the above and (\ref{eq:rho_bounds})
\begin{equation}\nonumber
  \frac{1}{n_i} \log |h'(T^{n_i+l_k(x)}(\hat{x}))| \leq \frac{1}{S(k+i+2)} \log (\lambda K W_2 |D_{k+i}|^{\alpha}).
\end{equation}
Taking limit as $i \longrightarrow \infty$,
\begin{equation}\nonumber
  \lim_{i \to \infty} \frac{1}{n_i} \log |h'(T^{n_i+l_k(x)}(\hat{x}))| \leq \lim_{i \to \infty} \frac{1}{S(k+i+2)} \log |D_{k+i}|^{\alpha}.
\end{equation}
Using Lemma \ref{lemm:diam_estimate}
\begin{eqnarray} \nonumber
  \lim_{i \to \infty} \frac{1}{S(k+i+2)} \log|D_{k+i}|^{\alpha}
  &=&\lim_{i \to \infty}\frac{1}{S(k+i+2)}\log \lambda^{- \alpha S(k+i+1)}   \\ \nonumber
  &=& -\alpha \log \lambda \lim_{i \to \infty} \frac{S(k+i+1)}{S(k+i+2)} \\ \nonumber
  &=& -\frac{\alpha}{\varphi} \log \lambda.
\end{eqnarray}
Then by (\ref{eq:lo_Lyap_exp_1}) \[ \chi_f^-(x) \leq 
(1- \frac{\alpha}{\varphi})\log \lambda < \log \lambda. \] 
% Finally, since 
% \[ (f^n)'(f^{l_k(x)}(x)) = 
% \frac{\prod_{i=0}^{n + l_k(x) -1} f'(f^i(x))}{\prod_{i=0}^{n+l_k(x)-1}f'(f^i(x))}, \]
% by \eqref{eq:up_Lyap_exp} and \eqref{eq:lo_Lyap_exp}, we have that 
% \[ \chi_f^+(f^{l_k(x)}(x)) = \chi_f^+(x) \hspace{0.5cm} \text{and} \hspace{0.5cm}
% \chi_f^-(f^{l_k(x)}(x)) = \chi_f^-(x).\]
This conclude the proof of the proposition.

\end{proof}

% \begin{rema}
% Observe that in the proof of 
% \end{rema}

%%% Local Variables:
%%% mode: latex
%%% TeX-master: "Unbounded_Lyapunov_exponent"
%%% End:

\section{Proof of Proposition \ref{prop:3}}
In this section we will give the proof of Proposition \ref{prop:3}. We will use the same notation as in the previous
sections. 

For every $k \geq 1$, put \[ D_k^+ \= (|c_{S(k)}|, c) \hspace{1cm} \text{ and } \hspace{1cm} D_k^- \= (-|c_{S(k)}|, c)\]
\[ A_k^+ \= D_k^+ \setminus \overline{D_{k+1}^+} \hspace{1cm} \text{ and } \hspace{1cm} 
A_k^- \= D_k^- \setminus \overline{D_{k+1}^-}. \] Observe that 
\begin{equation}
\label{eq:rec_obs}
 |A_k^+|=|D_k| - |D_{k+1}| = |A_k^-|. 
\end{equation}
\begin{lemm}
\label{lemm:rec_0}
There exist $\alpha''_+$, $\alpha''_-$, $\alpha'_+$, $\alpha'_-$, $K$, and $Q$ positive real numbers such that
\begin{equation}
    \label{eq:rec_lemm_+}
    \lambda^{-S(k)\alpha''_+}Q^{-1} \leq |h(A_k^+)| \leq \lambda^{-S(k)\alpha'_+}Q,
\end{equation}
and
\begin{equation}
    \label{eq:rec_lemm_-}
    \lambda^{-S(k)\alpha''_-}Q^{-1} \leq |h(A_k^-)| \leq \lambda^{-S(k)\alpha'_-}Q,
\end{equation}
for every $k \geq K$.
\end{lemm}

\begin{proof}
By Lemma \ref{lemm:diam_estimate}, there exists $\beta >0$ such that 
\[ \lim_{n \to \infty} \lambda^{S(k+1)}|D_k| = \beta. \] Let $\varepsilon >0$ be small enough so that 
$(\beta - \varepsilon)/(\beta + \varepsilon) \geq 1/2$. Let $M > 0$ be as in \eqref{eq:l_1} and \eqref{eq:l_2}. 
Fix $K > 0$ big enough so that, for every $k \geq K$, 
\eqref{eq:l_1}, \eqref{eq:l_2} holds on $A_k$, and the following holds:

\begin{equation}
    \label{eq:rec_1}
    \lambda^{-S(k+1)}(\beta - \varepsilon) \leq |D_k| \leq \lambda^{-S(k+1)}(\beta + \varepsilon),
\end{equation}

\begin{equation}
    \label{eq:rec_2}
     \lambda^{-S(k)} \leq \frac{1}{4},
\end{equation}
and
\begin{equation}
    \label{eq:rec_3}
     \frac{S(k+1)}{S(k)} < \varphi + \varepsilon.
\end{equation}
By \eqref{eq:rec_1}, with $k$ replaced by $k+1$, we get
\begin{equation}
    \label{eq:rec_7}
    \lambda^{S(k+2)} \frac{1}{\beta + \varepsilon}\leq \frac{1}{|D_{k+1}|} \leq \lambda^{S(k+2)} 
    \frac{1}{\beta - \varepsilon}.
\end{equation}
Combining \eqref{eq:rec_1} and  \eqref{eq:rec_7}, we get
\begin{equation}
    \label{eq:rec_8}
    \lambda^{S(k)} \frac{\beta - \varepsilon}{\beta + \varepsilon} \leq \frac{|D_{k}|}{|D_{k+1}|} \leq 
    \lambda^{S(k)} \frac{\beta + \varepsilon}{\beta- \varepsilon}.
\end{equation}
For $k \geq K$, using the mean value theorem on the function $h \colon A_k^+ \longrightarrow h(A_k^+)$, there exists
$\gamma^+ \in A_k^+$ such that
\begin{equation}
    \label{eq:rec_4}
     \frac{|h(A_k^+)|}{|A_k^+|} = |h'(\gamma^+)|.
\end{equation}
Let $\alpha^+$, be the right order of $0$ as a critical point of $h$. By \eqref{eq:l_1}, we have
\begin{equation}
    \label{eq:rec_5}
    e^{-M}|\gamma^+|^{\alpha^+} \leq |h'(\gamma^+)| \leq e^M|\gamma^+|^{\alpha^+}. 
\end{equation}
Since $\gamma^+ \in A_k^+$, we have that \[ |D_{k+1}| \leq |\gamma^+| \leq |D_k|.\] Then, by \eqref{eq:rec_1},
\eqref{eq:rec_2}, \eqref{eq:rec_4} and \eqref{eq:rec_5} we have that
\begin{equation}
    \label{eq:rec_6}
    e^{-M}|D_{k+1}|^{\alpha^++1} \left( \frac{|D_k|}{|D_{k+1}|} - 1 \right) \leq |h(A_k^+)| \leq 
    e^M|D_k|^{\alpha^+ + 1} \left( 1 - \frac{|D_{k+1}|}{|D_k|} \right).
\end{equation}
% In the same way, combining \eqref{eq:rec_1} and \eqref{eq:rec_7} with $k$ replaced by $k+1$, we get
% \begin{equation}
%     \label{eq:rec_9}
%     \lambda^{S(k)} \frac{\beta - \varepsilon}{ \beta + \varepsilon} \leq \frac{|D_k|}{|D_{k+1}|}.
% \end{equation}
Using \eqref{eq:rec_8} in \eqref{eq:rec_6}, we obtain
\begin{multline}
    \label{eq:rec_10}
    e^{-M}(\beta - \varepsilon)^{\alpha^+ + 1} \lambda^{-S(k+2)(\alpha^+ +1)} \left( \lambda^{S(k)} \frac{\beta -
    \varepsilon}{\beta + \varepsilon} - 1 \right) \leq  |h(A_k^+)| \leq \\ 
    e^M(\beta + \varepsilon)^{\alpha^+ + 1} \lambda^{-S(k+1)(\alpha^+ +1)} \left( 1 - \lambda^{-S(k)} \frac{\beta -
    \varepsilon}{\beta + \varepsilon}  \right).
\end{multline}
Put \[ Q_1 \= e^{-M} \left(\frac{\beta }{ 3}\right)^{\alpha^+ +1} \frac{1}{4} \hspace{1cm}  \text{ and } \hspace{1cm} 
Q_2 \= e^M 2 \beta. \] By \eqref{eq:rec_2} and since $(\beta - \varepsilon)/(\beta + \varepsilon) \geq 1/2$, we have that
\[ Q_1 \leq e^{-M}(\beta - \varepsilon)^{\alpha^+ + 1}  \left( \frac{\beta - \varepsilon}{\beta + \varepsilon} - 
\lambda^{-S(k)} \right), \] and 
\[ e^M(\beta + \varepsilon)^{\alpha^+ + 1} \left( 1 - \lambda^{-S(k)} \frac{\beta - \varepsilon}{\beta + \varepsilon} 
\right) \leq Q_2,\] for every $k \geq K$. Then 
\begin{equation}
    \label{eq:rec_11}
    \lambda^{-S(k+2)(\alpha^+ + 1)} \lambda^{S(k)} Q_1 \leq |h(A_k^+)| \leq \lambda^{-S(k+1)(\alpha^+ + 1)} Q_2.
\end{equation}
Finally, put \[\alpha'_+ \= \alpha^+ + 1 \hspace{0.5cm} \text{and} \hspace{0.5cm} \alpha''_+ \= (\varphi +
\varepsilon)^2(\alpha^+ + 1) - 1. \] Since 
$S(k) = S(k+2) - S(k-1)$, by \eqref{eq:rec_3} we have
\begin{eqnarray}
    -S(k+2)(\alpha^+ +1) + S(k) & =& -S(k) \left( \frac{S(k+2)}{S(k)}(\alpha^+ + 1) -1 \right) \nonumber \\
                            &\geq& - S(k)\alpha''_+. \nonumber
\end{eqnarray}
Then, taking $Q \= \max \{Q_1^{-1}, Q_2 \}$ we have
\[ \lambda^{-S(k)\alpha''_+} Q^{-1} \leq |h(A_k^+)| \leq \lambda^{S(k)\alpha'_+} Q.\]

In the same way we can prove \eqref{eq:rec_lemm_-}.
\end{proof}

\begin{proof}[Proof of Proposition \ref{prop:3}]
Let $\alpha''_+$, $\alpha''_-$, $\alpha'_+$, $\alpha'_-$, $K$, and $Q$, be as in Lemma \ref{lemm:rec_0}. By
\eqref{eq:rec_obs}, we have that 

\begin{equation}
    \label{eq:rec_12}
    \sum_{m=0}^{n}|h(A_{k+m}^+)| = |h(D_k^+)| - |h(D_{k+n+1}^+)|,
\end{equation}
for every $n \geq 0$. Then by \eqref{eq:rec_lemm_+} and \eqref{eq:rec_12} we get
\begin{equation}
    \label{eq:rec_13}
    Q^{-1}\sum_{m=0}^{n}\lambda^{-S(k+m)\alpha''_+} \leq |h(D_{k}^+)| - |h(D_{k+n+1}^+)| \leq Q \sum_{m=0}^n 
    \lambda^{-S(k+m)\alpha'_+},
\end{equation}
for every $n \geq 0$. Now, for $m \geq 0$ we have that 
% \begin{eqnarray}
%     S(k+m+1) -1 &=& \sum_{j=0}^{k+m-1}S(j) \nonumber \\
%                 &=& \sum_{j=0}^{k-1}S(j) = \sum_{j=k}^{k+m-1}S(j) \nonumber \\
%                 &=& S(k+1) -1 + \sum_{j=0}^{m-1}S(k+j). \nonumber
% \end{eqnarray}
% Thus
\begin{equation}
    \label{eq:rec_14}
    S(k+m) = S(k) + \sum_{j=0}^{m-1}S(k +j-1).
\end{equation}
Put
\[F_{k+m} \= \sum_{j=0}^{m-1}S(k+j-1), \hspace{0.5cm} 
%\sigma''_{m}(k) \= 1 + \sum_{i=0}^{m-1}\lambda^{-\alpha''_+ F_{k+i}}, \hspace{0.4cm} 
\text{and} \hspace{0.5 cm} \sigma'(k) \= 1 + \sum_{i=0}^{\infty}\lambda^{-\alpha'_+F_{k+i}}.\]
Then, combining \eqref{eq:rec_13} and \eqref{eq:rec_14} we obtain
\begin{equation}
    \label{eq:rec_15}
    \lambda^{-S(k)\alpha''_+}Q^{-1} \leq
    |h(D_k^+)|-|h(D_{k+n+1}^+)| \leq  \lambda^{-S(k)\alpha'_+} Q \sigma'(k).
\end{equation}
If we put
\[\Theta \= Q \left(1 + \sum_{i = 0}^{\infty} \lambda^{-\alpha''_+S(i)}\right), \] then for every $k \geq K$ and every 
$m \geq 0 $ we have
\[Q\sigma'(k) \leq \Theta, \hspace{0.5cm} \text{ and } \hspace{0.5cm}  \Theta^{-1} \leq Q^{-1}. \] 
Then
\begin{equation}
    \label{eq:rec_16}
    \lambda^{-S(k)\alpha''_+}\Theta^{-1} \leq
    |h(D_k^+)|-|h(D_{k+n+1}^+)| \leq  \lambda^{-S(k)\alpha'_+} \Theta.
\end{equation}
Since $|D_{k+n+1}| \longrightarrow 0$ as $n \longrightarrow \infty$, and $h$ is continuous, taking the limit in
\eqref{eq:rec_16} as $n \longrightarrow \infty$ we obtain
\begin{equation}
    \label{eq:rec_17}
    \lambda^{-S(k)\alpha''_+}\Theta^{-1}  \leq
    |h(D_k^+)| \leq  \lambda^{-S(k)\alpha'_+} \Theta.
\end{equation}
In the same way we can prove that 
\begin{equation}
    \label{eq:rec_18}
    \lambda^{-S(k)\alpha''_-}\Theta^{-1} \leq
    |h(D_k^-)| \leq  \lambda^{-S(k)\alpha'_-} \Theta.
\end{equation}
Finally, put $\alpha'' \= \max \{ \alpha''_- , \alpha''_+\}$, and $\alpha' \= \min \{ \alpha'_-, \alpha'_+\}$. For any
$k \geq K$ we have that \[ |f^{S(k)}(\tilde{c}) - \tilde{c}| = |h(D_k^+)| \hspace{1cm}  \text{or} \hspace{1cm} 
 |f^{S(k)}(\tilde{c}) - \tilde{c}| = |h(D_k^-)|. \] In any case, by \eqref{eq:rec_17} and \eqref{eq:rec_18}
the result follows.

\end{proof}
%%% Local Variables: 
%%% mode: latex
%%% TeX-master: "Unbounded_Lyapunov_exponent"
%%% End:

%\input{Appendix}

\bibliography{my_biblio}{}
\bibliographystyle{alpha}
\end{document}